\documentclass[12pt]{amsart}
\usepackage{amsmath,amsfonts,amsthm,amscd,amssymb,mathrsfs,amssymb}
\usepackage{stmaryrd}
\usepackage{graphicx}
\newtheorem{theorem}{Theorem}

\newtheorem{proposition}[theorem]{Proposition}

\newtheorem{definition}[theorem]{Definition}

\newtheorem{corollary}[theorem]{Corollary}

\newtheorem{remark}[theorem]{Remark}

\newcommand{\CP}{\mathbb{CP}}
\newcommand{\PP}{\mathbb{P}}
\newcommand{\CC}{\mathbb{C}}

\newcommand{\RR}{\mathbb{R}}

\newcommand{\U}{{\rm{U}}}

\newcommand{\z}{\mathbf{z}}

\numberwithin{equation}{section}
\numberwithin{theorem}{section}
\numberwithin{table}{section}
\numberwithin{table}{section}
\begin{document}
\bibliographystyle{amsalpha} 
\title[An equivalence of scalar curvatures]{An equivalence of scalar curvatures\\on Hermitian manifolds}
\author{Michael G. Dabkowski}
\address{Department of Mathematics, University of Michigan, Ann Arbor, 
MI, 48109}
\email{mgdabkow@umich.edu}
\author{Michael T. Lock}
\address{Department of Mathematics, University of Texas, Austin, 
TX, 78712}
\email{mlock@math.utexas.edu}
\thanks{The second author was partially supported by NSF Grant DMS-1148490}
\date{May 11, 2015}
\dedicatory{This article is dedicated Konstantin Leyzerovsky on the occasion of his birthday.}
\begin{abstract}
For a K\"ahler metric, the Riemannian scalar curvature is equal to twice the Chern scalar curvature.  The question we address here is whether this equivalence can hold for a non-K\"ahler Hermitian metric.  For such metrics, if they exist, the Chern scalar curvature would have the same geometric meaning as the Riemannian scalar curvature.
Recently, Liu-Yang showed that if this equivalence of scalar curvatures holds even in average over a compact Hermitian manifold, then the metric must in fact be K\"ahler.  However, we prove that a certain class of non-compact complex manifolds do admit Hermitian metrics for which this equivalence holds.  Subsequently, the question of to what extent the behavior of said metrics can be dictated is addressed and a classification theorem is proved.
\end{abstract}
\maketitle
\setcounter{tocdepth}{1}
\tableofcontents

\section{Introduction}
This article examines a relationship between scalar curvatures on Hermitian manifolds in the non-K\"ahler setting.  While there has been a plethora of work concerning the existence of certain classes of desirable K\"ahler metrics, there is markedly less understanding of what are ``choice'' non-K\"ahler Hermitian metrics.  For a K\"ahler manifold, the Chern connection on the holomorphic tangent bundle can be seen to coincide with Levi-Civita connection on the underlying Riemannian manifold, however, these can differ wildly for a non-K\"ahler Hermitian manifold.  
Our work here is inspired by the following broad question -- In the realm of non-K\"ahler Hermitian metrics on a complex manifold, do there exist some which induce a ``desirable relationship'' between the complex geometry and underlying real geometry?  Admittedly, this question is not only broad but also incredibly vague as it is not at all clear what ``desirable relationship'' should mean.  More specifically, what we wonder here, is whether some aspects of the respective geometries can coincide in the non-K\"ahler setting as they would in the K\"ahler setting.  One of the simplest such measures is that of scalar curvature.  On a Hermitian manifold, it is natural to consider the Riemannian and Chern scalar curvatures, i.e. those of the Levi-Civita connection on the underlying Riemannian manifold and the Chern connection respectively.  If the manifold is K\"ahler, then the Riemannian scalar curvature, $S$, is equal to twice the Chern scalar curvature, $S_C$.  This motivates the following definition.

\begin{definition}
{\em A Hermitian metric $g$, on a complex manifold $(M,J)$, is said to be a {\em K\"ahler like scalar curvature (Klsc) metric} if $S=2S_C$.} 
\end{definition}
From this, we ask whether a complex manifold can admit a non-K\"ahler Hermitian Klsc metric.  Since every Riemannian metric on a Riemann surface is K\"ahler, we are of course asking this for complex dimension $n\geq 2$. 
Formulas relating the Chern and Riemannian scalar curvatures on compact Hermitian manifolds were found in \cite{Gauduchon_scalar,Liu-Yang_2}, where it is subsequently shown that if $S$ and $2S_C$ are even equal in average over a compact manifold, which is clearly implied by Klsc, then the metric must in fact be K\"ahler.
Therefore, the actual question that we ask here is whether a non-compact complex manifold can admit
non-K\"ahler Hermitian Klsc metrics?

The Riemannian scalar curvature has a very clear geometric meaning in that it arises in the expansion of the volume of geodesic balls.  Specifically, for a Riemannian manifold $(M,g)$, the volume of the geodesic ball of radius $r$ around $p\in M$ has the asymptotic expansion
\begin{align}
\label{vol_expansion}
vol(B(p,r))=\omega_nr^n\Big(1-\frac{S(p)}{6(n+2)}r^2+\mathcal{O}(r^4)\Big),
\end{align}
where $\omega_n$ denotes the volume of the unit ball in $\RR^n$, see \cite{Besse}.  Hence, when the scalar curvature is positive or negative at a point, the volume of a small geodesic ball around that point is respectively smaller or larger than it would be for a Euclidean ball of the same radius.  If the metric is not K\"ahler, the Chern scalar curvature does not in general have such an obvious geometric interpretation.  However, if there were to exist a non-K\"ahler Hermitian Klsc metric, then its Chern scalar curvature would have such a meaning as $S(p)$ could be replaced with $2S_C(p)$ in the expansion \eqref{vol_expansion}.

In a different vein, there has been recent interest in the so called Chern-Yamabe problem.  Given a compact Hermitian manifold, this is the problem of finding a conformal deformation to a metric with constant Chern scalar curvature.  In \cite{Angella-Simone-Spotti}, it is proved that such conformal deformations exist in the case that the constant is non-positive.  (It is also worthwhile to note that an almost Hermitian version of the Yamabe problem was proved in \cite{delRio_Simanca}.)  One may ask when this problem overlaps with the usual Yamabe problem, which motivates the study of Klsc metrics.  In the compact case, recall that a Hermitian metric is a Klsc metric if and only if it is K\"ahler \cite{Gauduchon_scalar,Liu-Yang_2}.  In this paper, we study the Klsc condition in the non-compact $\U(n)$-invariant case, and both find and classify non-K\"ahler solutions.  In turn, this raises a very interesting question regarding the singular Yamabe problem and its Chern-scalar curvature analogue, see \eqref{singular_yamabe} in Section \eqref{questions} for more details.

We answer the question of existence of non-K\"ahler Hermitian Klsc metrics in the affirmative, by using conformal techniques, in Section \ref{conformal_method} below.  Interestingly, this shows that the obstruction in the compact case is global as opposed to local.  The idea here is to begin with a Hermitian metric and search for a non-K\"ahler Hermitian Klsc metric in its conformal class while keeping the complex structure fixed.  While this idea would be interesting to pursue in general, in order to make our work here more tractable, we have restricted our attention to the situation where the initial metric is a $\U(n)$-invariant K\"ahler metric.  This restriction not only enables us to find the desired metric in each conformal class explicitly, observe Theorem \ref{thm_1} below, but also lends itself to posing, and subsequently answering, questions involving the ability to specify the behavior of non-K\"ahler Klsc metrics as well as allows us to give a complete description of their moduli space.  These results are discussed in Section \ref{specifying}, where a detailed analysis of the regularity of these metrics is also given.

While background material and our conventions are discussed more thoroughly in Section \ref{background}, it is necessary to make a brief remark on notation here.  On $\CC^n$ we let $J_0$ denote the standard complex structure and define the radial variable $\z=\sum_{i=1}^{n}|z_i|^2$.  We will frequently consider Hermitian metrics on annuli of the form
\begin{align}
\mathcal{A}_{\alpha,\beta}=\{(z_1,\cdots,z_n)\in\CC^n:0\leq\alpha<\z<\beta\}.
\end{align}
Here $\sqrt{\alpha}$ and $\sqrt{\beta}$ are the radii of the annulus $\mathcal{A}_{\alpha,\beta}$.
Often, only the metric will be written, or it being non-K\"ahler Klsc referred to, and the fact that it is Hermitian with respect to $J_0$ will be implicit from context. 

\subsection{The conformal method}
\label{conformal_method}
In this section, we give an existence result for non-K\"ahler Hermitian Klsc metrics.
We begin by considering $\U(n)$-invariant K\"ahler metrics, i.e., those arising from some K\"ahler potential $\phi(\z)$, which is just a function of the radial variable $\z$,
on the annulus $\mathcal{A}_{\alpha,\beta}\subset (\CC^n,J_0)$.
Then, we seek a conformal factor, which can be understood in terms of the initial metric, that yields a non-K\"ahler Klsc metric.  Here the complex structure is fixed, so the resulting metric would be Hermitian with respect to $J_0$ as well.  In the following theorem, we see that such a conformal factor always exists.
\begin{theorem}
\label{thm_1}
Let $g$ be a $\U(n)$-invariant K\"ahler metric, with K\"ahler potential $\phi(\z)$, on the annulus $\mathcal{A}_{\alpha,\beta}\subset(\CC^n,J_0)$, for $n\geq 2$.  Then, the conformal metric $(v^2g, J_0)$, where
\begin{align}
\begin{split}
\label{conformal_factor}
\phantom{=}v=\Big(\int\frac{1}{\z^n(\phi')^{n-1}}d\z\Big)^{\frac{1}{2n-1}},
\end{split}
\end{align}
is a non-K\"ahler Hermitian Klsc metric.  Furthermore, up to scale and the constant of integration, this is the unique $\U(n)$-invariant conformal factor with this property.
\end{theorem}
This is proved in Section \ref{thm_1_proof}.

\begin{remark}
{\em Since the initial metric is K\"ahler, the constant of integration, which is suppressed in \eqref{conformal_factor}, can always be chosen so that the conformal factor does not vanish on the given annulus $\mathcal{A}_{\alpha,\beta}$.  In general, however, the constant of integration can cause the conformal factor to vanish on the sphere of radius $\gamma$ for at most one isolated value of $\z=\gamma\in (\alpha,\beta)$.  (It can cause it to vanish at most once because $\phi$ is a K\"ahler potential on the annulus.)  In this case, the annulus splits into two sub-annuli, $\mathcal{A}_{\alpha,\beta}=\mathcal{A}_{\alpha,\gamma}\amalg\mathcal{A}_{\gamma,\beta}$, and a non-K\"ahler Hermitian Klsc metric is obtained on each.}
\end{remark}

\begin{remark}
{\em
Although a compact Hermitian manifold cannot admit a non-K\"ahler Klsc metric, recall \cite{Gauduchon_scalar,Liu-Yang_2} as mentioned above, Theorem \ref{thm_1} can still be applied in the compact setting in order to obtain non-K\"ahler Klsc metrics on an annular decomposition of the underlying manifold minus closed sets of positive codimension.  Here, by the compact setting, we
we mean a compact manifold obtained by completing $\CC^n\setminus\{0\}$ by adding in a point or a
$\CP^{n-1}$ at $\{0\}$ and $\{\infty\}$ with a K\"ahler metric coming from extending a $\U(n)$-invariant K\"ahler metric on $\CC^n$ over these compactifications, e.g., the Fubini-Study metric on $\CP^n$.  See Section \ref{Fubini-Study} for an example.
}
\end{remark}

\subsection{Moduli space and behavior specification}
\label{specifying}
Here, we determine the moduli space of $\U(n)$-invariant non-K\"ahler Hermitian Klsc metrics on an annulus as well as address the question of 
to what extent the behavior of such metrics can be dictated.  Since the standard metric on $S^{2n-1}$ decomposes into the Fubini-Study metric on $\CP^{n-1}$ and metric along the Hopf fiber, which we denote by 
$g_{\CP^{n-1}}$ and $h$ respectively, any $\U(n)$-invariant metric on the annulus $\mathcal{A}_{\alpha,\beta}\subset\CC^n$ can be written
as
\begin{align}
\label{U(n)-invariant_g}
g=\mathscr{E}(\z)\cdot\Big((d\sqrt{\z})^2+\z h\Big)+\z \mathscr{F}(\z)\cdot\big(g_{\CP^{n-1}}\big),
\end{align}
where $\mathscr{E}(\z)$ and $\mathscr{F}(\z)$ are smooth positive functions of the radial variable $\z$ on $\mathcal{A}_{\alpha,\beta}$.  Therefore, we ask whether it is possible to specify such an $\mathscr{E}(\z)$ or $\mathscr{F}(\z)$ and obtain a non-K\"ahler Hermitian Klsc metric of the form \eqref{U(n)-invariant_g}.  We will see,  in Theorem \ref{thm_2} below, that if the specified function satisfies a certain broad admissibility condition, then fixing one of either $\mathscr{E}(\z)$ or $\mathscr{F}(\z)$ along with a constant determines a unique such metric.  From this, we will then be able to completely describe the moduli space.  For certain analytic reasons, which will be apparent in the proofs, we will choose to fix such an $\mathscr{F}(\z)$ and recover an $\mathscr{E}(\z)$.  We begin with an admissibility requirement for the specified function,  which will ensure that we do, in fact, obtain metrics in Theorem~\ref{thm_2}, Theorem \ref{thm_3} and Corollay \ref{cor_thm_3} below.

\begin{definition}
{\em Let $\mathscr{F}(\z)$ be a smooth positive function on the annulus $\mathcal{A}_{\alpha,\beta}\subset\CC^n$, and fix a constant $\mathcal{C}$.  The pair $(\mathscr{F},\mathcal{C})$ is said to be {\em admissible} on $\mathcal{A}_{\alpha,\beta}$ if
\begin{align*}
\z\mathscr{F}\cdot\Big(\frac{1}{2n-1} \int\frac{1}{\z^n \mathscr{F}^{n-1}}d\z+\mathcal{C}\Big)^{-2}
\end{align*}
is strictly increasing there.}
\end{definition}
\begin{remark}
{\em Although this admissibility condition may seem unusual, intuitively it can be thought of as the weakening of the condition that $(\z \mathscr{F})'>0$, see Remark~\ref{admissible_remark} for a more in depth explanation.  In turn, we will see that there is a tremendously wide range of functions which satisfy this admissibility condition.
}
\end{remark}

Now, we are able to answer the question of specifying the behavior of the non-K\"ahler Klsc metric posed above, as well as completely describe the moduli space.

\begin{theorem}
\label{thm_2}
To each admissible pair $(\mathscr{F},\mathcal{C})$ on the annulus $\mathcal{A}_{\alpha,\beta}\subset~(\CC^n,J_0)$, for $n\geq 2$, there exists a unique associated $\U(n)$-invariant non-K\"ahler Hermitian Klsc metric given by
\begin{equation*}
g_{(\mathscr{F},\mathcal{C})}=\Bigg((\z\mathscr{F})'-\frac{2}{\z^{n-1}\mathscr{F}^{n-2}}\Big(\int\frac{1}{\z^n\mathscr{F}^{n-1}}d\z+\mathcal{C}\Big)^{-1}\Bigg)\Big((d\sqrt{\z})^2+\z h\Big)+(\z \mathscr{F})\cdot g_{\CP^{n-1}}.
\end{equation*}
Moreover, the moduli space of $\U(n)$-invariant non-K\"ahler Hermitian Klsc metrics on $(\mathcal{A}_{\alpha,\beta},J_0)$ is given by all such metrics as $(\mathscr{F},\mathcal{C})$ ranges over all admissible pairs there.
\end{theorem}

This is proved in Section \ref{pf_thm_2}.  In other words, what we find here is that the moduli space of $\U(n)$-invariant non-K\"ahler Hermitian Klsc metrics on an annulus is equivalent to the space of admissible pairs there.  Each $\mathscr{F}(\z)\in C^{\infty}(\mathcal{A}_{\alpha,\beta})$, that has some constant with which it forms an admissible pair, in fact forms an admissible pair exactly with an open interval of constants (this is seen easily by examining the admissibility condition).  To each such $\mathscr{F}(\z)$, there is an associated $1$-parameter family of non-K\"ahler Klsc metrics in the moduli space given by $g_{(\mathscr{F},\mathcal{C})}$ as $\mathcal{C}$ ranges over that interval of constants for which the pair $(\mathscr{F},\mathcal{C})$ is admissible.  Therefore, the moduli space can be viewed as a collection of these $1$-parameter families of metrics over the set of functions which have the ability to be admissible.  

Given this classification of the $\U(n)$-invariant non-K\"ahler Hermitian Klsc metrics on an annulus, it is reasonable to ask about their boundary behavior.  In particular, do the spaces admitting these non-K\"ahler Klsc metrics necessarily have two open ends?  In the following result, we see that if the specified $\mathscr{F}(\z)$ is smooth on all of $\CC^n$, then we obtain a non-K\"ahler Klsc metric on $\CC^n$ with a canonical asymptotically cone-like singularity at the origin in that, near the origin, after a change of variables it looks like a canonical model plus higher order terms (which we abbreviate by {\em h.o.t.}).  After this change of variables, however, the coefficients of these higher order terms may not always extend smoothly over the origin.  We describe the regularity of such a metric by saying that it is of class $C^{\ell}$, where $C^{\ell}$ is the lowest regularity that any coefficient of the higher order terms has, in the appropriate variable, over the origin.

\begin{theorem}
\label{thm_3}
Let $(\mathscr{F},\mathcal{C})$ be an admissible pair on the annulus $\mathcal{A}_{0,\infty}=\CC^n\setminus\{0\}$, for $n\geq 2$, where $\mathscr{F}(\z)$ is a smooth function on $\CC^n$ whose first nonzero term in the Taylor expansion about the origin is of order $k$.  Consider the uniquely associated
$\U(n)$-invariant non-K\"ahler Hermitian Klsc metric $g_{(\mathscr{F},\mathcal{C})}$ of Theorem \ref{thm_2}.
\begin{enumerate}
\item The metric $g_{(\mathscr{F},\mathcal{C})}$ extends to $(\CC^n,J_0)$ with a canonical asymptotically cone-like singularity at the origin in the sense that there exists a change of variables so that near the origin the metric is given by
\begin{align*}
g_{(\mathscr{F},\mathcal{C})}=dr^2+\Big(\frac{k+1}{2n-1}\Big)r^2\Big(g_{\CP^{n-1}}+(k+1)(2n-1)h\Big)+h.o.t..
\end{align*}
\item
Furthermore, the class of this extension over the origin is described as follows:
\begin{itemize}
\item When $k=0$ or $k=1$, this metric is precisely of class
\begin{equation*}
\begin{cases}
C^{\infty}\phantom{=}&\text{if}\phantom{=}\frac{\partial^{(k+1)(n-1)}}{\partial\z^{(k+1)(n-1)}}\big(\frac{\z^k}{\mathscr{F}}\big)^{n-1}\big|_{\z=0}=0\\
C^{2n-3}\phantom{=}&\text{otherwise}
\end{cases}.
\end{equation*}
\item When $k\geq 2$, this metric is, a priori, of class $C^0$.  
\end{itemize}
\end{enumerate}
\end{theorem}
This is proved in Section \ref{pf_thm_3}.
\begin{remark}
\label{regularity_remark}
{\em When $k\geq 2$, the metrics of Theorem \ref{thm_3} can only, a priori, be guaranteed to be of class $C^0$ over the origin, however, in general such a metric may be of a class with much higher regularity.
This will depend upon the lowest fractional power of the variable $r$ that appears in a certain expansion of the coefficients of the higher order terms, and upon whether the derivative $\frac{\partial^{(k+1)(n-1)}}{\partial\z^{(k+1)(n-1)}}\big(\frac{\z^k}{\mathscr{F}}\big)^{n-1}\big|_{\z=0}$ vanishes.  When $k\geq2$, it is actually possible for such a metric to be of class $C^{\infty}$ as well.
See Section~\ref{regularity_proof}, in particular Remark \ref{k>2Cinfty}, for a thorough explanation, and Section \ref{fix_2} for an example.
}
\end{remark}

The singularity at the origin of the non-K\"ahler Klsc metrics,
obtained in Theorem~\ref{thm_3}, depends only upon the complex dimension and  the order of the first nonzero term of the Taylor expansion of $\mathscr{F}$ at the origin.  In particular, notice that if $\mathscr{F}(0)>0$, so $k=0$, then then the singularity
is determined by the complex dimension alone.
Similarly, given certain relationships between $k$ and $n$, we can obtain non-K\"ahler Klsc metrics which are asymptotic to a cone over a scaled $S^{2n-1}$, and in some cases even nonsingular non-K\"ahler Klsc metrics.  Specifically, we have the following corollary. 
\begin{corollary}
\label{cor_thm_3}
Let $(\mathscr{F},\mathcal{C})$ be an admissible pair on the annulus $\mathcal{A}_{0,\infty}=\CC^n\setminus~\{0\}$, for $n\geq 2$, where $\mathscr{F}(\z)$ is a smooth function on $\CC^n$ whose first nonzero term in the Taylor expansion about the origin is of order $k$.  If $(k+1)(2n-1)=\ell^2$, for some positive integer $\ell$, then the uniquely associated $\U(n)$-invariant non-K\"ahler Hermitican Klsc metric $g_{(\mathscr{F},\mathcal{C})}$ on $(\CC^n, J_0)$, of Theorem~\ref{thm_3},
descends via a $\mathbb{Z}/\ell\mathbb{Z}$ quotient to
\begin{align*}
\overline{g}_{(\mathscr{F},\mathcal{C})}=dr^2+\Big(\frac{k+1}{2n-1}\Big)r^2g_{S^{2n-1}}+h.o.t..
\end{align*}
In particular, if $k=2n-2$, then 
\begin{align*}
\overline{g}_{(\mathscr{F},\mathcal{C})}=dr^2+r^2g_{S^{2n-1}}+h.o.t.
\end{align*}
is a nonsingular Klsc metric.  
\end{corollary}
This is proved in Section \ref{pf_cor_thm_3}.  Note that the regularity statements of Theorem \ref{thm_3} hold for these metrics as well and, in particular, when $k=2n-2$ the nonsingular Klsc metric $\overline{g}_{(\mathscr{F},\mathcal{C})}$ is, a priori, of class $C^0$ over the origin.  However, once again, these often will have higher regularity, recall Remark \ref{regularity_remark}.  See Section \ref{fix_2} for an example.

\begin{remark}
{\em
Both Theorem \ref{thm_3} and Corollary \ref{cor_thm_3} can be restated to hold for the pair $(\mathscr{F},\mathcal{C})$ being admissible on the annulus $\mathcal{A}_{0,\beta}\subset\CC^n$, for any $\beta>0$, as opposed to on $\mathcal{A}_{0,\infty}=\CC^n\setminus\{0\}$.  In this case, the resulting Klsc metric, $g_{(\mathscr{F},\mathcal{C})}$, will be defined on the open ball of radius $\sqrt{\beta}$ about the origin and the same statements concerning the singularity and regularity hold.  The statements of Theorem \ref{thm_3} and Corollary \ref{cor_thm_3} are given for all of $\CC^n$ only for simplicity of presentation.}
\end{remark}

\subsection{Questions}
\label{questions}
There are several natural questions motivated by our work here:
\begin{enumerate}
\item The conformal method used to proved Theorem \ref{thm_1} is simplified by restricting to the $\U(n)$-invariant setting as it reduces the question from a quasi-linear PDE to a non-linear ODE.  Does a generalization of Theorem \ref{thm_1} hold when no symmetry is assumed?  If so, what can then be said about generalizations of the results of Section \ref{specifying}?

\item Every $\U(n)$-invariant Hermitian metric on an annulus $\mathcal{A}_{\alpha,\beta}\subset(\CC^n,J_0)$ is conformal to a $\U(n)$-invariant K\"ahler metric there.  However, this is not necessarily true if the $\U(n)$-invariant assumption is dropped.  The fact that we begin with a K\"ahler metric and then search for the desired conformal factor simplifies the proof of Theorem \ref{thm_1} since $S=2S_C$  for such an initial metric, which leads to cancelation in \eqref{eq_1}, and also because the Laplacian for a K\"ahler metric here takes the form \eqref{laplacian}, which reduces \eqref{eq_2} to \eqref{eq_3}.  If a generalization of Theorem \ref{thm_1} exists when no symmetry is assumed, could it be further generalized to the situation where the initial metric is a non-conformally-K\"ahler Hermitian metric?

\item As mentioned above, the Klsc condition gives a clear and geometric meaning to the Chern scalar curvature as it can replace the Riemannian scalar curvature in the expansion of the volume of geodesic balls \eqref{vol_expansion}.  Does the Klsc condition have other geometric implications, or force restrictions of the scalar curvature, on non-compact non-K\"ahler Hermitian manifolds?

\item
\label{singular_yamabe} The singular Yamabe problem is stated as follows:  Given a compact Riemannian manifold $(M,g)$ and a closed subset $K\subset M$, find a complete metric $\tilde{g}$ on $M\setminus K$ which has constant scalar curvature and is conformal to $g$.
While this problem is not solved completely, there has been considerable progress, for instance, see \cite{LoewnerNirenberg,Schoen_singularYamabe,SchoenYau_singularYamabe,AvilesMcOwen,Mazzeo_singularYamabe,MazzeoPacard_sYamabe1,MazzeoPacard_sYamabe2,Finn1,Finn2}.  The Chern-Yamabe problem generalizes naturally to the singular case as well.  To the best of the authors knowledge, this has not been studied.  From our work here, we know that there exist non-K\"ahler Klsc metrics in the non-compact setting.  Thus, we ask whether there are non-K\"ahler Hermitian Klsc metrics which simultaneously solve both the singular Yamabe and singular Chern-Yamabe problems?
\end{enumerate}

\subsection{Acknowledgements}
The authors would like to thank Jeff Viaclovsky for his insight into this problem and many useful comments, as well as Ethan Leeman and Matthew Gursky for many valuable discussions.

\section{Background and conventions}
\label{background}
In this section we provide a brief description of the complex geometry necessary for this work as well as fix our conventions and notation.  In particular, we discuss Hermitian metrics, the Chern connection and the curvature of the Chern connection.  For a more thorough exposition of this material see \cite{Huybrechts}.

Let $(M,J)$ be a complex $n$-dimensional manifold.  The complexified tangent bundle decomposes as
\begin{align}
TM\otimes \CC=T^{1,0}M\oplus T^{0,1}M,
\end{align}
into the $\pm\sqrt{-1}$ eigenspaces of $J$ respectively.  In terms of a local holomorphic coordinate system $\{z_1,\cdots,z_n\}$, where $z_i=x_i+\sqrt{-1}y_i$, locally these eigenspaces are
\begin{align}
\begin{split}
T^{1,0}M=\text{span}\Big\{\frac{\partial}{\partial z_1}, \cdots, \frac{\partial}{\partial z_n}\Big\}\phantom{=}\text{and}\phantom{=}T^{0,1}M=\text{span}\Big\{\frac{\partial}{\partial \bar{z}_1}, \cdots, \frac{\partial}{\partial \bar{z}_n}\Big\},
\end{split} 
\end{align}
where
\begin{align}
\frac{\partial}{\partial z_i}=\frac{1}{2}\Big(\frac{\partial}{\partial x_i}-\sqrt{-1}\frac{\partial}{\partial y_i}\Big)\phantom{=}\text{and}\phantom{=}\frac{\partial}{\partial \bar{z}_i}=\frac{1}{2}\Big(\frac{\partial}{\partial x_i}+\sqrt{-1}\frac{\partial}{\partial y_i}\Big).
\end{align}
The complexified cotangent bundle decomposes into the $\pm\sqrt{-1}$ eigenspaces of $J$ analogously as
\begin{align}
T^*M\otimes \CC=\Lambda^{1,0}M\oplus \Lambda^{0,1}M,
\end{align}
which are locally given by
\begin{align}
\begin{split}
\Lambda^{1,0}M=\text{span}\big\{dz_1, \cdots, dz_n\big\}\phantom{=}\text{and}\phantom{=}
\Lambda^{0,1}M=\text{span}\big\{d\bar{z}_1, \cdots, d\bar{z}_n\big\},
\end{split} 
\end{align}
where
\begin{align}
dz_i=dx_i+\sqrt{-1}dy_i\phantom{=}\text{and}\phantom{=}d\bar{z}_i=dx_i-\sqrt{-1}dy_i.
\end{align}

On a complex manifold $(M,J)$, a Riemannian metric $g$ is Hermitian if 
\begin{align}
g(JX,JY)=g(X,Y)
\end{align}
for all vector fields $X$ and $Y$.  Given a Hermitian metric $g$, the associated $2$-form $\omega\in\Lambda^{1,1}$ is defined by $\omega(\cdot,\cdot)=g(J\cdot,\cdot)$ and can be written as 
\begin{align}
\omega=\sqrt{-1}\sum_{i,j}g_{i\bar{j}}dz_i\wedge d\bar{z}_j,
\end{align}
where $g_{i\bar{j}}=g(\frac{\partial}{\partial z_i},\frac{\partial}{\partial \bar{z}_j})$.  We will use a Hermitian metric and its associated $2$-form interchangeably.  A Hermitian metric $g$ is K\"ahler if the associated $2$-form $\omega$ is closed, and in this case $\omega$ is called the K\"ahler form.

The Chern connection, on the Hermitian holomorphic tangent bundle $(T^{1,0}M,g)$, is the unique connection that is compatible with both the Hermitian metric and complex structure.  This is opposed to the Levi-Civita connection which is the unique symmetric (torsion free) connection that is compatible with the Riemannian metric.  There is a natural isomorphism between $T^{1,0}M$ and $TM$ under which the Chern connection (or, more generally, any Hermitian connection) corresponds to a connection on the underlying Riemannian manifold.  However, the Chern connection may have a nonvanishing torsion tensor, and the induced connection is not necessarily the Levi-Civita connection.  In fact, under this isomorphism, the Chern and Levi-Civita connections coincide if and only if the metric is K\"ahler.  Note that in the non-K\"ahler setting, although these connections differ, one can still look for relationships between their respective induced geometries which is precisely the focus of this paper. 

We conclude this section with a short discussion on the curvature tensor with respect to the Chern connection, and for explicit curvature formulas we refer the reader to \cite{Liu-Yang_1}.  Write the components of the full curvature tensor with respect to the Chern connection as
$\Theta_{i\bar{j}k\bar{\ell}}$.  Taking the trace on the $k$ and $\bar{\ell}$ indices yields the components of the Chern-Ricci curvature
\begin{align}
\Theta_{i\bar{j}}=g^{k\bar{\ell}}\Theta_{i\bar{j}k\bar{\ell}},
\end{align}
which has the associated Chern-Ricci form
\begin{align}
\rho=\sqrt{-1}\Theta_{i\bar{j}}dz_i\wedge d\bar{z}_j=-\sqrt{-1}\partial\bar{\partial}\log\det(g_{i\bar{j}}).
\end{align}
Then, by taking the trace of the Chern-Ricci curvature, we obtain the Chern scalar curvature
\begin{align}
S_C=g^{i\bar{j}}\Theta_{i\bar{j}}=g^{i\bar{j}}g^{k\bar{\ell}}\Theta_{i\bar{j}k\bar{\ell}}.
\end{align} 
Lastly, we note that on a complex $n$-dimensional Hermitian manifold $(M,J,g)$, the Chern scalar-curvature, Chern-Ricci form and the associated $2$-form satisfy the relationship
\begin{align}
\label{S_C_relationship_1}
\frac{S_C}{n}\omega^n=\rho\wedge\omega^{n-1}.
\end{align}
This will be useful in finding how the Chern scalar curvature transforms under a conformal change in the metric in Section~\ref{Chern_scalar_change}.

\section{The conformal method}
\label{thm_1_proof}
In this section we prove Theorem \ref{thm_1}, the idea of which is as follows.  
Starting with a K\"ahler metric, we examine how both the Riemannian and Chern scalar curvatures change under conformal transformation.  The transformation of the Riemannian scalar curvature under a conformal change in the metric is well-known.  Let $(M,g)$ be a $2n$-dimensional Riemannian manifold, where $n\geq 2$, with Riemannian scalar curvature $S$.  Then the conformal metric $\tilde{g}=u^{\frac{2}{n-1}}g$ has Riemannian scalar curvature 
\begin{align}
\label{riemannian_change}
\widetilde{S}=-2\frac{2n-1}{n-1}u^{-\frac{n+1}{n-1}}\Delta u+Su^{-\frac{2}{n-1}}.
\end{align}
While a general formula for the transformation of the Chern scalar curvature under a conformal change is known, our interest is in how it transforms specifically in the $\U(n)$-invariant setting when the initial metric is K\"ahler.  We prove such a formula below in Section \ref{Chern_scalar_change}.  This target simplification is essential to the rest of the proof. 
Then, in Section \ref{finding_u}, we equate the formulas for a conformal change in the respective scalar curvatures, with the appropriate scale factor, to obtain a quasi-linear PDE which reduces to a non-linear second order ODE from the $\U(n)$-invariance assumption.  Finally, we solve this to find an explicit solution for a conformal factor, which is unique up to scale and constant of integration, that results in a non-K\"ahler Hermitian Klsc metric.  

Given the complexity of the equations, it is somewhat surprising that we were able to find an explicit solution. Having such an explicit solution is very useful as it allows for the results of Section \ref{specifying}.  Also, it is important to remark here that the fact that the initial metric is K\"ahler will be of considerable aid the analysis.

\subsection{The conformal transformation of Chern scalar curvature}
\label{Chern_scalar_change}
Let $(\omega,J_0)$ be a $\U(n)$-invariant K\"ahler metric on a spherical shell about the origin in $\CC^n$, where
\begin{align}
\omega=\sqrt{-1}\partial\bar{\partial}\phi(\z)
\end{align}
for some K\"ahler potential $\phi(\z)$.
The volume form is given by
\begin{align}
\frac{\omega^n}{n!}=Vdz_1\wedge d\bar{z}_1\wedge\cdots\wedge dz_n\wedge d\bar{z}_n,
\end{align}
and the Chern-Ricci form by
\begin{align}
\rho=-\sqrt{-1}\partial\bar{\partial}\log(V).
\end{align}

Consider the conformal metric  $(\widetilde{\omega},J_0)$, where $\widetilde{\omega}=u^{\frac{2}{n-1}}\omega$.  Recalling \eqref{S_C_relationship_1}, since this conformal metric is Hermitian with respect to the complex structure $J_0$, its Chern scalar curvature, $\widetilde{S_C}$, and Chern-Ricci form, $\widetilde{\rho}$, satisfy
\begin{align}
\label{S_C_relationship_2}
\Bigg(\frac{\widetilde{S_C}}{n}\Bigg)\widetilde{\omega}^n=\widetilde{\rho}\wedge\widetilde{\omega}^{n-1}.
\end{align}
We will use \eqref{S_C_relationship_2} to extract the transformation of the Chern scalar curvature under a conformal change.

First, observe that since
\begin{align}
\begin{split}
\label{conformal_omega}
\widetilde{\omega}^{n-1}=&u^2\omega^{n-1}\\
\text{and}\phantom{====}&\\
\widetilde{\omega}^n=&u^{\frac{2n}{n-1}}\omega^n,
\end{split}
\end{align}
the Chern-Ricci form transforms as
\begin{align}
\label{conformal_rho}
\widetilde{\rho}=-\sqrt{-1}\partial\bar{\partial}\log\Big(u^{\frac{2n}{n-1}}V\Big)=-\sqrt{-1}\Big(\frac{2n}{n-1}\Big)\partial\bar{\partial}\log(u)+\rho.
\end{align}
Combining \eqref{S_C_relationship_2}, \eqref{conformal_omega} and \eqref{conformal_rho} we find that
\begin{align}
\Bigg(\frac{\widetilde{S_C}}{n}\Bigg)\widetilde{\omega}^n=u^2\Bigg(-\sqrt{-1}\Big(\frac{2n}{n-1}\Big)\partial\bar{\partial}\log(u)\wedge\omega^{n-1}+\frac{S_C}{n}\omega^n\Bigg).
\end{align}

Now, because of the $\U(n)$-invariance, without loss of generality we can restrict our attention to the $z_1$-axis where the K\"ahler form is given by
\begin{align}
\label{omega_z1}
\omega=\sqrt{-1}\partial\bar{\partial}\phi=\sqrt{-1}\Bigg((\z\phi')'dz_1\wedge d\bar{z}_1+\sum_{i=2}^n\phi' dz_i\wedge d\bar{z}_i\Bigg)
\end{align}
and, following from \eqref{conformal_rho}, the Chern-Ricci form of the conformal metric by
\begin{align}
\widetilde{\rho}=-\sqrt{-1}\frac{2n}{n-1}\Bigg(\Big(\z\frac{u'}{u}\Big)'dz_1\wedge d\bar{z}_1+\sum_{i=2}^n\frac{u'}{u} dz_i\wedge d\bar{z}_i\Bigg)+\rho.
\end{align}
On the $z_1$-axis, from \eqref{omega_z1}, observe that 
\begin{align}
\begin{split}
\label{omega_n_n-1}
\omega^n=&\Big((-1)^{\frac{n}{2}} n! (\phi')^{n-1}(\z\phi')'\Big)dz_1\wedge d\bar{z}_1\wedge \cdots \wedge dz_n\wedge d\bar{z}_n\\
\text{and}\phantom{==}&\\
\omega^{n-1}=&(-1)^{\frac{n-1}{2}}(n-1)!\Bigg((\phi')^{n-1}dz_2\wedge d\bar{z}_2\wedge \cdots \wedge dz_n\wedge d\bar{z}_n\\
&\phantom{========}+(\phi')^{n-2}(\z\phi')'\sum _{i=2}^ndz_1\wedge d\bar{z}_1\wedge \cdots \wedge\widehat{dz_i\wedge d\bar{z}_i}\cdots\wedge dz_n\wedge d\bar{z}_n\Bigg),
\end{split}
\end{align}
where $\widehat{dz_i\wedge d\bar{z}_i}$ denotes the omission of that wedge product.  

Therefore, from \eqref{conformal_rho} and \eqref{omega_n_n-1}, we find that
\begin{align}
\begin{split}
\label{chern_change_wedge}
\frac{\widetilde{S_C}}{n}\widetilde{\omega}^n
=(-1)^{\frac{n}{2}}n!u^2(\phi')^{n-2}\Bigg[&-\Big(\frac{2}{n-1}\Big)\Bigg(\Big(\z\frac{u'}{u}\Big)'\phi'+(n-1)\Big(\frac{u'}{u}\Big)(\z\phi')'\Bigg)\\
&+\frac{S_C}{n}\phi'(\z\phi')'\Bigg]dz_1 \wedge d\bar{z}_1\wedge\cdots \wedge dz_n\wedge  d\bar{z}_n.
\end{split}
\end{align}
Finally, from \eqref{conformal_omega}, \eqref{omega_n_n-1} and \eqref{chern_change_wedge}, we find that, under such a conformal change, the Chern scalar curvature transforms as
\begin{align}
\label{chern_change}
\widetilde{S_C}=-\frac{2n}{n-1}\Big(\frac{u^{-\frac{2}{n-1}}}{\phi'(\z\phi')'}\Big)\Bigg(\Big(\z\frac{u'}{u}\Big)'\phi'+(n-1)\Big(\frac{u'}{u}\Big)(\z\phi')'\Bigg)+u^{-\frac{2}{n-1}}S_C.
\end{align}

\subsection{The Klsc conformal factor}
\label{finding_u}
Let $(\omega,J_0)$ be a $\U(n)$-invariant K\"ahler metric on a spherical shell about the origin in $\CC^n$ with K\"ahler potential $\phi(\z)$.  Since $(\omega,J_0)$  is K\"ahler, we have that the Riemannian and Chern scalar curvatures here satisfy
\begin{align}
\label{initial_equivalence}
S=2S_C.
\end{align} 
We would like to find a conformal factor $u^{\frac{2}{n-1}}$ so that
\begin{align}
\widetilde{S}=2\widetilde{S_C},
\end{align}
Therefore, we divide \eqref{riemannian_change} by $2$ and set it equal to \eqref{chern_change} to obtain the equation
\begin{align}
\begin{split}
\label{eq_1}
-\frac{2n-1}{n-1}u^{-\frac{n+1}{n-1}}\Delta u+u^{-\frac{2}{n-1}}\Big(\frac{S}{2}\Big)=&-\frac{(2n)u^{-\frac{2}{n-1}}}{(n-1)\phi'(\z\phi')'}\Bigg(\Big(\z\frac{u'}{u}\Big)'\phi'+(n-1)\Big(\frac{u'}{u}\Big)(\z\phi')'\Bigg)\\
&+u^{-\frac{2}{n-1}}S_C.
\end{split}
\end{align}
It is important to note that, since the initial metric is K\"ahler, from \eqref{initial_equivalence}, the terms $u^{-\frac{2}{n-1}}(S/2)$ and $u^{-\frac{2}{n-1}}S_C$, on the left and right hand sides of \eqref{eq_1} respectively, cancel.  This is the first reduction made possible due to the fact that the initial metric is K\"ahler.  

Equation \eqref{eq_1} then further reduces to the following non-linear second order ODE:
\begin{align}
\label{eq_2}
\Big(\frac{\phi'(\z\phi')'}{u}\Big)\Delta u=\Big(\frac{2n}{2n-1}\Big)\Bigg(\Big(\z\frac{u'}{u}\Big)'\phi'+(n-1)\Big(\frac{u'}{u}\Big)(\z\phi')'\Bigg).
\end{align}
We exploit that the initial metric is K\"ahler once again to simplify \eqref{eq_2}.  For a general K\"ahler metric $(g,J)$, the Riemannian Laplacian is given by $\Delta(\cdot)=2g^{i\bar{j}}\frac{\partial^2}{\partial z_i\partial \bar{z}_j}(\cdot)$.
Therefore, here
\begin{align}
\label{laplacian}
\Delta u=2\Bigg(\frac{(\z u')'}{(\z\phi')'}+(n-1)\frac{u'}{\phi'}\Bigg),
\end{align}
hence \eqref{eq_2} reduces to
\begin{align}
\label{eq_3}
\frac{1}{u}\Big((\z u')'\phi'+(n-1)u'(\z \phi')'\Big)=\frac{n}{2n-1}\Bigg(\Big(\z\frac{u'}{u}\Big)'\phi'+(n-1)\Big(\frac{u'}{u}\Big)(\z\phi')'\Bigg).
\end{align}
Next, multiply \eqref{eq_3} through by $\z$, and change variables by letting
\begin{align}
\begin{split}
y=&\frac{u'}{u},\phantom{=}\text{so }y'+y^2=\frac{u''}{u},\\
\text{and}\phantom{==}&\\
f=&\z\phi'.
\end{split}
\end{align}
This reduces the equation to
\begin{align}
\label{eq_4}
fy+\z f(y'+y^2)+(n-1)\z f'y =\frac{n}{2n-1}\Big(fy+\z fy'+(n-1)  \z f'y\Big).
\end{align}
Then, isolate the $y^2$ term to obtain the equation
\begin{align}
\label{eq_5}
\z f y^2=-\Big(\frac{n-1}{2n-1}\Big)\Big(f+(n-1)\z f'\Big)y-\Big(\frac{n-1}{2n-1}\Big)\z fy'.
\end{align}
From multiplying \eqref{eq_4} through by $-\frac{2n-1}{n-1}\big(\z f y^2\big)^{-1}$, we obtain
\begin{align}
\label{eq_6}
-\frac{2n-1}{n-1}=\Big(\frac{f+(n-1)\z f'}{\z f}\Big)\frac{1}{y}+\frac{y'}{y^2}.
\end{align}
Change variables once again by setting
\begin{align}
w=\frac{1}{y},\phantom{=}\text{so }w'=-\frac{y'}{y^2}.
\end{align}
Then, \eqref{eq_6} becomes
\begin{align}
\label{eq_7}
w'-\Big(\frac{f+(n-1)\z f'}{\z f}\Big)w=\frac{2n-1}{n-1}.
\end{align}
By using the integrating factor
\begin{align}
\exp\Big(-\int \frac{f+(n-1)\z f'}{\z f} d\z\Big)=\frac{1}{\z f^{n-1}},
\end{align}
equation \eqref{eq_7} can be solved to find
\begin{align}
\label{eq_8}
w=\frac{2n-1}{n-1}\big(\z f^{n-1}\big)\int\frac{1}{\z f^{n-1}}d\z.
\end{align}
Here, we suppress the constant of integration.
Now, since $w=\frac{u}{u'}$, by inverting both sides of \eqref{eq_8} and integrating, we find that
\begin{align}
\label{eq_9}
\log(u)=\int\frac{1}{w}d\z=\frac{n-1}{2n-1}\int\frac{1}{\z f^{n-1}\big(\int\frac{1}{\z f^{n-1}}d\z\big)}d\z.
\end{align}
To integrate the right hand side of \eqref{eq_9}, we make the substitution
\begin{align}
h=\int\frac{1}{\z f^{n-1}}d\z,\phantom{=}\text{so}\phantom{=}dh=\frac{1}{\z f^{n-1}}d\z,
\end{align}
and see that
\begin{align}
\label{eq_10}
\log(u)=\frac{n-1}{2n-1}\int\frac{1}{h}dh=\frac{n-1}{2n-1}\log\Big|\int\frac{1}{\z f^{n-1}}d\z\Big|+C,
\end{align}
for some constant $C$.
Finally, by exponentiating, we see that
\begin{align}
u=C\Big|\int\frac{1}{\z f^{n-1}}d\z\Big|^{\frac{n-1}{2n-1}}=C\Big|\int\frac{1}{\z^n(\phi')^{n-1}}d\z\Big|^{\frac{n-1}{2n-1}}.
\end{align}
Therefore, up to scale, the desired Klsc conformal factor is
\begin{align}
v^2=u^\frac{2}{n-1}=\Big(\int\frac{1}{\z^n(\phi')^{n-1}}d\z\Big)^{\frac{2}{2n-1}}.
\end{align}

\begin{remark}
{\em
The fact that the conformal metric $(v^2\omega,J_0)$ is no longer K\"ahler follows immediately from examining its associated $2$-form restricted to the $z_1$-axis because if it were K\"ahler there would exist some function $\psi(\z)$ so that $(\z\psi')'=v^2(\z\phi')'$ and $\psi'=v^2\phi'$, but such a $\psi(\z)$ cannot exist since $v$ is non-constant.  Therefore, $(v^2\omega,J_0)$ is a non-K\"ahler Hermitian Klsc metric.
}
\end{remark}

\section{Moduli space and behavior specification}
\label{pf_thm_2}
In this section we prove Theorem \ref{thm_2}.
We would like to show that there is a $\U(n)$-invariant non-K\"ahler Hermitian Klsc metric where the coefficient of the $g_{\CP^{n-1}}$ component is given by $\z\mathscr{F}$ which, given a further specified constant, is unique.   The proof will proceed in three steps.  First, we will show that, from such a given function $\mathscr{F}(\z)$, we can recover a $\phi(\z)$ so that $\mathscr{F}=\big(\int\frac{1}{\z^n(\phi')^{n-1}}d\z\big)^{\frac{2}{2n-1}}\phi'$.  We would like this $\phi(\z)$ to be a K\"ahler potential, so next, we show that if the pair $(\mathscr{F},\mathcal{C})$ is admissible on that annulus, then the recovered function $\phi(\z)$ is in fact a K\"ahler potential there.  Finally, this is used to obtain the other component of the metric.

\begin{proposition}
\label{phi'_prop}
Let $\mathscr{F}(\z)$ be a smooth positive function on the annulus $\mathcal{A}_{\alpha,\beta}\subset{\CC^n}$.  Then, the function $\phi'(\z)$ defined by
\begin{align}
\label{phi_from_G}
\phi'(\z)=\mathscr{F}\cdot\Big(\frac{1}{2n-1}\int\frac{1}{\z^n\mathscr{F}^{n-1}}d\z+\mathcal{C}\Big)^{-2}
\end{align}
is such that
\begin{align}
\label{G_from_phi}
\mathscr{F}(\z)=\phi'\cdot\Big(\int\frac{1}{\z^n(\phi')^{n-1}}d\z\Big)^{\frac{2}{2n-1}}.
\end{align}
Furthermore, $\phi'(\z)$ is the unique such function up to the constant $\mathcal{C}$.
\end{proposition}

\begin{proof}
Begin by observing that
\begin{align}
\z^n\mathscr{F}^{n-1}=\Big(\int\frac{1}{\z^n(\phi')^{n-1}}d\z\Big)^{\frac{2(n-1)}{2n-1}}\z^n(\phi')^{n-1},
\end{align}
since $\mathscr{F}=\big(\int\frac{1}{\z^n(\phi')^{n-1}}d\z\big)^{\frac{2}{2n-1}}\phi'$, and thus that
\begin{align}
\label{int_1/zF_1}
\int\frac{1}{\z^n\mathscr{F}^{n-1}}d\z=\int\frac{1}{\z^n(\phi')^{n-1}\big(\int\frac{1}{\z^n(\phi')^{n-1}}d\z\big)^{\frac{2(n-1)}{2n-1}}}d\z.
\end{align}
To integrate the right hand side of \eqref{int_1/zF_1}, we make the substitution
\begin{align}
h=\int\frac{1}{\z^n(\phi')^{n-1}}d\z,\phantom{=}\text{so}\phantom{=}dh=\frac{1}{\z^n(\phi')^{n-1}}d\z,
\end{align}
and find that
\begin{align}
\label{int_1/zF_2}
\int\frac{1}{\z^n\mathscr{F}^{n-1}}d\z+\mathcal{\widetilde{C}}=(2n-1)h^{\frac{1}{2n-1}}=(2n-1)\Big(\int\frac{1}{\z^n(\phi')^{n-1}}d\z\Big)^{\frac{1}{2n-1}},
\end{align}
where $\mathcal{\widetilde{C}}=(2n-1)\mathcal{C}$ is the constant of integration.  Raising the terms of \eqref{int_1/zF_2} to the $(2n-1)^{th}$-power yeilds
\begin{align}
\Big(\frac{1}{2n-1}\int\frac{1}{\z^n\mathscr{F}^{n-1}}d\z+\mathcal{C}\Big)^{2n-1}=\int\frac{1}{\z^n(\phi')^{n-1}}d\z.
\end{align}
Taking a derivative, we see that
\begin{align}
\frac{1}{\z^n\mathscr{F}^{n-1}}\Big(\frac{1}{2n-1}\int\frac{1}{\z^n\mathscr{F}^{n-1}}d\z+\mathcal{C}\Big)^{2n-2}=\frac{1}{\z^n(\phi')^{n-1}}.
\end{align}
Finally, we are able to solve for $\phi'(\z)$ in terms of $\mathscr{F}(\z)$ as
\begin{align}
\phi'(\z)=\mathscr{F}\cdot\Big(\frac{1}{2n-1}\int\frac{1}{\z^n\mathscr{F}^{n-1}}d\z+\mathcal{C}\Big)^{-2}.
\end{align}
\end{proof}

The function $\phi'(\z)$ recovered in Proposition \ref{phi'_prop} is possibly the derivative of a K\"ahler potential $\phi(\z)$.  Since we are in the $\U(n)$-invariant setting, in order for $\phi(\z)$ to be a K\"ahler potential on the annulus $\mathcal{A}_{\alpha,\beta}$, it is necessary for both $\phi'>0$ and $(\z\phi')'>0$ there.  In the following proposition, we see exactly when this is satisfied.

\begin{proposition}
\label{admissible_proposition}
Let $\mathscr{F}(\z)$ be a smooth positive function on the annulus $\mathcal{A}_{\alpha,\beta}\subset\CC^n$, and let $\mathcal{C}$ be a fixed constant.  If the function 
\begin{align}
\label{(G,C)_increasing_2}
\z \mathscr{F}\cdot\Big(\frac{1}{2n-1}\int\frac{1}{\z^n \mathscr{F}^{n-1}}d\z+\mathcal{C}\Big)^{-2}
\end{align}
is strictly increasing on $\mathcal{A}_{\alpha,\beta}$, then the function $\phi(\z)$ defined by
\begin{align}
\label{phi_from_G_2}
\phi(\z)=\int\frac{\mathscr{F}}{\big(\frac{1}{2n-1}\int\frac{1}{\z^n\mathscr{F}^{n-1}}d\z+\mathcal{C}\big)^{-2}}d\z
\end{align}
is a K\"ahler potential there.
\end{proposition}

\begin{proof}
A function $\phi(\z)$ is a K\"ahler potential on the annulus $\mathcal{A}_{\alpha,\beta}$ if and only if both
$\phi'>0$ and $(\z\phi')'>0$ there.  First, note that since $\mathscr{F}>0$, if \eqref{(G,C)_increasing_2} is strictly increasing, then $\phi'>0$ on the annulus.  Then, observe, from \eqref{phi_from_G} and \eqref{phi_from_G_2}, that the inequality $(\z\phi')'>0$ is equivalent to the inequality
\begin{align}
\label{admissible_inequality}
\frac{\partial}{\partial\z}\Bigg(\z \mathscr{F}\cdot\Big(\frac{1}{2n-1}\int\frac{1}{\z^n\mathscr{F}^{n-1}}d\z+\mathcal{C}\Big)^{-2}\Bigg)>0
\end{align}
from which the strictly increasing condition of \eqref{(G,C)_increasing_2} follows.
\end{proof}

\begin{remark}
\label{admissible_remark}
{\em We call a pair $(\mathscr{F},\mathcal{C})$ {\em admissible} if the recovered function $\phi(\z)$, given by \eqref{phi_from_G_2}, is a K\"ahler potential as then $\mathscr{F}(\z)$ arises from a K\"ahler potential as a part of the conformal Klsc metric.
Also, although the admissibility condition may seem complicated, it in fact can be viewed as in some sense as a weakening of the condition that $(\z \mathscr{F})'>0$.  This is seen as follows.  If the constant of $\mathcal{C}=0$, expanding \eqref{admissible_inequality} and simplifying yields the inequality
\begin{align}
\label{admissible_inequality_2}
(\z\mathscr{F})'-\frac{2}{\z^{n-1}\mathscr{F}^{n-2}}\Big(\int\frac{1}{\z^n\mathscr{F}^{n-1}}d\z\Big)^{-1}>0.
\end{align}
Then, observe that if $(\z\mathscr{F})'>0$, the anti-derivative $\int\frac{1}{\z^n\mathscr{F}^{n-1}}<0$ for nearly all such $\mathscr{F}$ and therefore \eqref{admissible_inequality_2} is satisfied.  It is important to note from this, that pairs $(\mathscr{F},\mathcal{C})$ satisfying the admissibility condition are plentiful.}
\end{remark}

We have thus far shown that, given an admissible pair $(\mathscr{F},\mathcal{C})$ on an annulus $\mathcal{A}_{\alpha,\beta}$, we can recover a K\"ahler potential $\phi(\z)$.  Let $(\omega, J_0)$ denote the K\"ahler form of the K\"ahler metric arising from this potential, and consider the corresponding conformal non-K\"ahler Klsc metric as given by Theorem \ref{thm_1}.  Because of the $\U(n)$-invariance, we can restrict our attention to examining the associated $2$-form of this metric, $(\widetilde{\omega},J_0)$, on the $z_1$-axis which from \eqref{omega_z1} and Theorem \ref{thm_1}, we see is given by
\begin{align}
\label{hermitian_form_1}
\widetilde{\omega}=\sqrt{-1}\Big(\int\frac{1}{\z^n(\phi')^{n-1}}d\z\Big)^{\frac{2}{2n-1}}\Bigg((\z\phi')'dz_1\wedge d\bar{z}_2+\phi'\sum_{j=2}^ndz_j\wedge d\bar{z}_j\Bigg).
\end{align}
From Proposition \ref{phi'_prop}, we see both that $\mathscr{F}=\phi'\cdot\Big(\int\frac{1}{\z^n(\phi')^{n-1}}d\z\Big)^{\frac{2}{2n-1}}$ and that
\begin{align}
\label{F/phi'}
\Big(\int\frac{1}{\z^n(\phi')^{n-1}}d\z\Big)^{\frac{2}{2n-1}}=\frac{\mathscr{F}}{\phi'}=\Big(\frac{1}{2n-1}\int\frac{1}{\z^n\mathscr{F}^{n-1}}d\z+\mathcal{C}\Big)^{2},
\end{align}
from which we find that
\begin{align}
\begin{split}
(\z\phi')'\cdot\Big(&\int\frac{1}{\z^n(\phi')^{n-1}}d\z\Big)^{\frac{2}{2n-1}}\\
&=\Big(\frac{1}{2n-1}\int\frac{1}{\z^n\mathscr{F}^{n-1}}d\z+\mathcal{C}\Big)^{2}\cdot\frac{\partial}{\partial\z}\Bigg(\z \mathscr{F}\cdot\Big(\frac{1}{2n-1}\int\frac{1}{\z^n\mathscr{F}^{n-1}}d\z+\mathcal{C}\Big)^{-2}\Bigg)\\
&=(\z \mathscr{F})'-\frac{2}{\z^{n-1}\mathscr{F}^{n-2}}\Big(\int\frac{1}{\z^n\mathscr{F}^{n-1}}d\z+\mathcal{C}\Big)^{-1}.
\end{split}
\end{align}
Therefore, the associated $2$-form of the conformal Klsc metric \eqref{hermitian_form_1} is given by
\begin{align}
\begin{split}
\label{hermitian_form_2}
\widetilde{\omega}=\sqrt{-1}\Bigg(\Bigg[(\z \mathscr{F})'-\frac{2}{\z^{n-1}\mathscr{F}^{n-2}}\Big(\int\frac{1}{\z^n\mathscr{F}^{n-1}}d\z+\mathcal{C}\Big)^{-1} \Bigg]&dz_1\wedge d\bar{z}_1\\
+\mathscr{F}\sum_{j=2}^n&dz_j\wedge d\bar{z}_j\Bigg).
\end{split}
\end{align}
Recall that the standard metric on $S^{2n-1}$ is decomposed into $g_{\CP^{n-1}}$, the Fubini-Study metric on $\CP^{n-1}$, and $h$, the metric along the Hopf fiber.   Therefore, since $\omega(\cdot,\cdot)=g(J_0\cdot, \cdot)$ and we are in the $\U(n)$-invariant setting, from \eqref{hermitian_form_2} we see that the non-K\"ahler Hermitian Klsc metric is explicitly given by
\begin{align}
\label{gFC_metric}
g_{(\mathscr{F},\mathcal{C})}=\Bigg((\z\mathscr{F})'-\frac{2}{\z^{n-1}\mathscr{F}^{n-2}}\Big(\int\frac{1}{\z^n\mathscr{F}^{n-1}}d\z+\mathcal{C}\Big)^{-1}\Bigg)\Big((d\sqrt{\z})^2+\z h\Big)+(\z \mathscr{F})g_{\CP^{n-1}}.
\end{align}

Lastly, we obtain the uniqueness and classification result of Theorem \ref{thm_2} as follows.  First, note that due to the admissibility condition, the coefficient of the $\big((d\sqrt{\z})^2+\z h\big)$ component of the metric $g_{(\mathscr{F},\mathcal{C})}$ from \eqref{gFC_metric} is smooth on the annulus.  Next, we will use the fact that any $\U(n)$-invariant Hermitian metric on an annulus $\mathcal{A}_{\alpha,\beta}\subset (\CC^n,J_0)$ is conformal to a K\"ahler metric there, and furthermore that, up to scale, there exists a unique $\U(n)$-invariant K\"ahler metric in each conformal class.  This is seen as follows.  Let $\Omega$ be the associated $2$-form to a $\U(n)$-invariant non-K\"ahler Hermitian metric on $\mathcal{A}_{\alpha,\beta}\subset(\CC^n,J_0)$.  Then, there exists some $\mathscr{E}(\z),\mathscr{F}(\z)\in C^{\infty}(\mathcal{A}_{\alpha,\beta})$ so that the restriction of $\Omega$ to the $z_1$-axis is given by
\begin{align}
\Omega=\sqrt{-1}\Big(\mathscr{E}(\z)dz_1\wedge d\bar{z}_1+\mathscr{F}(\z)\sum_{i=2}^ndz_i\wedge d\bar{z}_i\Big).
\end{align}
We seek some function $v(\z)\in C^{\infty}(\mathcal{A}_{\alpha,\beta})$ so that the conformal metric $e^v\Omega$ is K\"ahler, which would mean that there exists some potential function $\phi(\z)$ so that
\begin{align}
e^v\mathscr{E}=(\z\phi')'\phantom{=}\text{and}\phantom{=}e^v\mathscr{F}=\phi',
\end{align}
where $\phi'$ is the derivative of the K\"ahler potential.
Substituting the second equation into the first, we find that
\begin{align}
\label{second_conformal}
v=\int\frac{\mathscr{E}-(\z \mathscr{F})'}{\z \mathscr{F}} d\z\phantom{=}\text{and}\phantom{=}\phi'=\exp\Big(\int\frac{\mathscr{E}-\mathscr{F}}{\z \mathscr{F}}d\z\Big),
\end{align}
and hence that there is a unique, up to scale, $\U(n)$-invariant K\"ahler metric in each conformal class.
Finally, suppose this non-K\"ahler Hermitian metric $(\mathcal{A}_{\alpha,\beta},J_0,\Omega)$ is a Klsc metric.  We will show that it is necessarily of the form $g_{(\mathscr{F},\mathcal{C})}$ in \eqref{gFC_metric}.
By equating the formulas \eqref{phi_from_G} and \eqref{second_conformal} for $\phi'(\z)$ and taking a derivative, we find that
\begin{align}
\begin{split}
\label{equate_derivative}
\Big(\frac{1}{2n-1}\int\frac{1}{\z^n\mathscr{F}^{n-1}}d\z+\mathcal{C}_1\Big)^{-2}\Bigg(\mathscr{F}'&-\frac{2}{\z^n\mathscr{F}^{n-2}}\Big(\int\frac{1}{\z^n\mathscr{F}^{n-1}}d\z+\mathcal{C}_1\Big)^{-1}\Bigg)\\
=&\Big(\frac{\mathscr{E}-\mathscr{F}}{\z \mathscr{F}}\Big)\exp\Big(\int\frac{\mathscr{E}-\mathscr{F}}{\z \mathscr{F}}d\z+\mathcal{C}_2\Big).
\end{split}
\end{align}
Next, by substituting the formula \eqref{phi_from_G} for the term $\exp\big(\int\frac{\mathscr{E}-\mathscr{F}}{\z \mathscr{F}}d\z+\mathcal{C}_2\big)$ on the right hand side of \eqref{equate_derivative}, since they are equivalent expressions of $\phi'(\z)$, and making appropriate cancellations, we obtain the equation
\begin{align}
\mathscr{F}'-\frac{2}{\z^n\mathscr{F}^{n-2}}\Big(\frac{1}{2n-1}\int\frac{1}{\z^n\mathscr{F}^{n-1}}d\z+\mathcal{C}_1\Big)^{-1}=\frac{\mathscr{E}-\mathscr{F}}{\z}.
\end{align}
This can be solved for $\mathscr{E}$ to find that
\begin{align}
\mathscr{E}=(\z\mathscr{F})'-\frac{2}{\z^{n-1}\mathscr{F}^{n-2}}\Big(\frac{1}{2n-1}\int\frac{1}{\z^n\mathscr{F}^{n-1}}d\z+\mathcal{C}_1\Big)^{-1},
\end{align} 
which is precisely the same formula as the coefficient of the $\big((d\sqrt{\z})^2+\z h\big)$ component of $g_{(\mathscr{F},\mathcal{C})}$ in \eqref{gFC_metric}.  Therefore, any non-K\"ahler Hermitian Klsc metric on $(\mathcal{A}_{\alpha,\beta},J_0)$ must take the form $g_{(\mathscr{F},\mathcal{C})}$ for some admissible pair.

\section{Asymptotically canonical singularities}
\label{pf_thm_3}
In this section we prove Theorem \ref{thm_3} and Corollary \ref{cor_thm_3}, the former of which is constituted by Section \ref{expansion_proof} and Section \ref{regularity_proof}, and the latter by Section \ref{pf_cor_thm_3}.

Let $(\mathscr{F},\mathcal{C})$ be an admissible pair on the annulus $\mathcal{A}_{\alpha,\beta}\subset\CC^n$, where 
the Taylor expansion of $\mathscr{F}(\z)\in C^{\infty}(\CC^n)$ about the origin is given by
\begin{align}
\label{G_expansion}
\mathscr{F}(\z)=\sum_{j=k}^{\infty}\frac{\mathscr{F}^{(j)}(0)}{j!}\z^j.
\end{align}
Note that the first nonzero term is of order $k$.  From Theorem \ref{thm_2}, we have that 
\begin{align}
g_{(\mathscr{F},\mathcal{C})}=\Bigg((\z \mathscr{F})'-\frac{2}{\z^{n-1}\mathscr{F}^{n-2}}\Big(\int\frac{1}{\z^n\mathscr{F}^{n-1}}d\z+\mathcal{C}\Big)^{-1}\Bigg)\Big((d\sqrt{\z})^2+\z h\Big)+(\z \mathscr{F})g_{\CP^{n-1}}
\end{align}
is the corresponding unique Klsc metric on $\mathcal{A}_{0,\infty}$.  In order to understand the behavior of $g$ near the origin, it is necessary to analyze the asymptotic behavior of the terms which comprise the coefficients of $\big((d\sqrt{\z})^2+\z h\big)$ and $g_{\CP^{n-1}}$ as they approach $\z=0$.  In Section \ref{expansion_proof}, we prove the first statement of Theorem \ref{thm_3}, that the Klsc metric extends to all of $\CC^n$ and has a canonical asymptotically cone-like singularity at the origin.  Then, in Section \ref{regularity_proof}, we prove the regularity conditions at the origin.

Throughout the proof a variety of expansions will be considered and manipulated, and it will only be necessary to keep track of the coefficients of the lowest order terms.  Therefore, for brevity we denote the coefficients of the higher order terms of the expansions with capital Latin letters indexed with respect to the appropriate term in the expansion.

\subsection{The expansion of the Klsc metric}
\label{expansion_proof}
We begin by examining the coefficient of $\big((d\sqrt{\z})^2+\z h\big)$.  To study the integral term, first observe from \eqref{G_expansion}, the expansion of $\mathscr{F}(\z)$, that
\begin{align}
\label{G^n-1}
\mathscr{F}^{n-1}=\Big(\frac{\mathscr{F}^{(k)}(0)}{k!}\Big)^{n-1}\z^{k(n-1)}+\sum_{j=1}^{\infty}A_j\z^{k(n-1)+j},
\end{align}
where the constants $A_j$ are the appropriate combination of the coefficients of the Taylor expansion of $\mathscr{F}(\z)$.  Then, around $\z=0$, we have the expansion
\begin{align}
\frac{\z^{k(n-1)}}{\mathscr{F}^{n-1}}=\frac{\z^{k(n-1)}}{\big(\frac{\mathscr{F}^{(k)}(0)}{k!}\big)^{n-1}\z^{k(n-1)}+\sum_{j=1}^{\infty}A_j\z^{k(n-1)+j}}=\Big(\frac{\mathscr{F}^{(k)}(0)}{k!}\Big)^{1-n}\Big(1+\sum_{j=1}^{\infty}B_j\z^j\Big),
\end{align}
where $B_j=\big(\frac{\mathscr{F}^{(k)}(0)}{k!}\big)^{n-1}\big(\frac{1}{j!}\big)\frac{\partial^j}{\partial\z^j}\big(\frac{\z^{k(n-1)}}{\mathscr{F}^{n-1}}\big)\big|_{\z=0}$, and thus
\begin{align}
\frac{1}{\z^n\mathscr{F}^{n-1}}=\Big(\frac{1}{\z^{k(n-1)+n}}\Big)\Big(\frac{\z^{k(n-1)}}{\mathscr{F}^{n-1}}\Big)=\Big(\frac{\mathscr{F}^{(k)}(0)}{k!}\Big)^{1-n}\Big(\frac{1}{\z^{k(n-1)+n}}\Big)\Big(1+\sum_{j=1}^{\infty}B_j\z^j\Big).
\end{align}
Therefore, in a small neighborhood of $\z=0$, we find that
\begin{align}
\begin{split}
\label{int_1/zG}
\int&\frac{1}{\z^n\mathscr{F}^{n-1}}d\z=\Big(\frac{\mathscr{F}^{(k)}(0)}{k!}\Big)^{1-n}\int \frac{1+\sum_{j=1}^{\infty}B_j\z^j}{\z^{k(n-1)+n}}d\z\\
&\phantom{=}=-\Big(\frac{\mathscr{F}^{(k)}(0)}{k!}\Big)^{1-n}\Bigg(\frac{1}{(n-1)(k+1)\z^{(n-1)(k-1)}}+\sum_{j=1}^{(n-1)(k+1)-1}C_j\z^{-(n-1)(k+1)+j}\\
&\phantom{============}-B_{(n-1)(k+1)}\log(\z)-\sum_{j=(n-1)(k+1)+1}^{\infty}C_j\z^{-(n-1)(k+1)+j}\Bigg),
\end{split}
\end{align}
where $C_j=\frac{B_j}{(n-1)(k+1)-j}$ for $j\neq (n-1)(k+1)$.  Notice that the coefficient of the $\log(\z)$ term is still just $B_{(n-1)(k+1)}$  From this, we see that
\begin{align}
\begin{split}
\label{int_1/zG^-1} 
\Big(\int\frac{1}{\z^n\mathscr{F}^{n-1}}d\z+\mathcal{C}\Big)^{-1}=&-\frac{(n-1)(k+1)\big(\frac{\mathscr{F}^{(k)}(0)}{k!}\big)^{n-1}\z^{(n-1)(k+1)}}{1+\sum_{j=1}^{\infty}D_j\z^j-E\z^{(n-1)(k+1)}\log(\z)},
\end{split}
\end{align}
where the constants
\begin{align}
\begin{split}
\label{D,E_coeffiecients}
D_j&=\begin{cases}
(n-1)(k+1)C_j\phantom{=}&j< (n-1)(k+1)\\
\mathcal{C}(n-1)(k+1)\big(\frac{\mathscr{F}^{(k)}(0)}{k!}\big)^{n-1}\phantom{=}&j=(n-1)(k+1)\\
-(n-1)(k+1)C_j\phantom{=}&j> (n-1)(k+1)
\end{cases}\\
\text{and}\phantom{=}&\\
E&=(n-1)(k+1)B_{(n-1)(k+1)}.
\end{split}
\end{align}

Next, similarly to \eqref{G^n-1}, we find that
\begin{align}
\label{z^n-1G^n-2}
\z^{n-1}\mathscr{F}^{n-2}=\z^{k(n-2)+n-1}\Big(\frac{\mathscr{F}^{(k)}(0)}{k!}\Big)^{n-2}\Big(1+\sum_{j=1}^{\infty}F_j\z^j\Big),
\end{align}
for some appropriately defined constants $F_j$.

Therefore, from \eqref{int_1/zG^-1} and \eqref{z^n-1G^n-2}, we see that in some neighborhood of $\z=0$
\begin{align}
\begin{split}
\label{coefficient_1_expansion_0}
-\frac{2}{\z^{n-1}\mathscr{F}^{n-2}}&\Big(\int\frac{1}{\z^n\mathscr{F}^{n-1}}d\z+\mathcal{C}\Big)^{-1}\\
=&\Bigg(\frac{2\big(\frac{\mathscr{F}^{(k)}(0)}{k!}\big)^{2-n}}{\z^{k(n-2)+n-1}\big(1+\sum_{j=1}^{\infty}F_j\z^j\big)}\Bigg)\Bigg(\frac{(n-1)(k+1)\big(\frac{\mathscr{F}^{(k)}(0)}{k!}\big)^{n-1}\z^{(n-1)(k+1)}}{1+\sum_{j=1}^{\infty}D_j\z^j+E\z^{(n-1)(k+1)}\log(\z)}\Bigg)\\
=&\frac{2(n-1)(k+1)\big(\frac{\mathscr{F}^{(k)}(0)}{k!}\big)\z^k}{1+\sum_{j=1}^{\infty}G_j\z^j+\z^{(n-1)(k+1)}\log(\z)\sum_{j=0}^{\infty}H_j\z^j},
\end{split}
\end{align}
for some appropriately defined constants $G_j$ and $H_j$.  Then, letting 
\begin{align}
\label{G(z)}
\mathcal{G}(\z)=\sum_{j=1}^{\infty}G_j\z^j+\z^{(n-1)(k+1)}\log(\z)\sum_{j=0}^{\infty}H_j\z^j,
\end{align}
observe that, around $\z=0$, we have the following expansion
\begin{equation}
\label{1/zF_int}
-\frac{2}{\z^{n-1}\mathscr{F}^{n-2}}\Big(\int\frac{1}{\z^n\mathscr{F}^{n-1}}d\z+\mathcal{C}\Big)^{-1}=2(n-1)(k+1)\Big(\frac{\mathscr{F}^{(k)}(0)}{k!}\Big)\z^k\Big(1+\sum_{j=1}^{\infty}(-1)^{j}\mathcal{G}(\z)^j\Big).
\end{equation}
Note that, since there are no constant terms in $\mathcal{G}(\z)$, the lowest order term of \eqref{1/zF_int} is just $2(n-1)(k+1)\big(\frac{\mathscr{F}^{(k)}(0)}{k!}\big)\z^k$.

The other term in the coefficient of $\big((d\sqrt{\z})^2+\z h\big)$ is $(\z\mathscr{F})'$.  From \eqref{G_expansion}, we see that this has the expansion
\begin{align}
\label{zF'_expansion}
(\z \mathscr{F})'=\sum_{j=k}^{\infty}(j+1)\frac{\mathscr{F}^{(j)}(0)}{j!}\z^j
\end{align}
around $\z=0$.

We use the expansions \eqref{1/zF_int} and \eqref{zF'_expansion} together to find the expansion of the coefficient of $\big((d\sqrt{\z})^2+\z h\big)$ around $\z=0$.  It will be important to keep track of the coefficient of the lowest order term, and since the lowest order terms in both \eqref{1/zF_int} and \eqref{zF'_expansion} are constant multiples of $\z^k$, we combine the higher order terms as
\begin{align}
\label{higher_order}
\mathcal{H}(\z)=2(n-1)(k+1)\big(\frac{\mathscr{F}^{(k)}(0)}{k!}\big)\z^k\sum_{j=1}^{\infty}(-1)^j\mathcal{G}(\z)^j+\sum_{j=k+1}^{\infty}(j+1)\frac{\mathscr{F}^{(j)}(0)}{j!}\z^j.
\end{align}
Note that the lowest possible order term of $\mathcal{H}(\z)$ is a constant multiple of $\z^{k+1}$.
Finally, around the origin, we find that the coefficient of $\big((d\sqrt{\z})^2+\z h\big)$
\begin{align}
\label{coefficient_1_expansion}
(\z \mathscr{F}(\z))'-\frac{2}{\z^{n-1}\mathscr{F}^{n-2}}\Big(\int\frac{1}{\z^n\mathscr{F}^{n-1}}d\z+\mathcal{C}\Big)^{-1}=&(2n-1)(k+1)\Big(\frac{\mathscr{F}^{(k)}(0)}{k!}\Big)\z^k+\mathcal{H}(\z).
\end{align} 

Now, we examine the expansion of $\z\mathscr{F}$, the coefficient of $g_{\CP^{n-1}}$.  From \eqref{G_expansion}, we see that
\begin{align}
\label{coefficient_2_expansion}
\z\mathscr{F}=\sum_{j=k}^{\infty}\frac{\mathscr{F}^{(j)}(0)}{j!}\z^{j+1},
\end{align}
around $\z=0$, and notice that the lowest order term here is just $\big(\frac{\mathscr{F}^{(k)}(0)}{k!}\big)\z^{k+1}$.

Using \eqref{coefficient_1_expansion} and \eqref{coefficient_2_expansion}, we find that, around $\z=0$, the metric $g$ has the following expansion.
\begin{align}
\begin{split}
\label{metric_expansion_1}
g_{(\mathscr{F},\mathcal{C})}=(2n-1)(k+1)\Big(&\frac{\mathscr{F}^{(k)}(0)}{k!}\Big)\z^k\Big((d\sqrt{\z})^2+\z h\Big)+\Big(\frac{\mathscr{F}^{(k)}(0)}{k!}\Big)\z^{k+1}g_{\CP^{n-1}}\\
&+\mathcal{H}(\z)\Big((d\sqrt{\z})^2+\z h\Big)+\Big(\sum_{j=k+1}^{\infty}\frac{\mathscr{F}^{(j)}(0)}{j!}\z^{j+1}\Big)g_{\CP^{n-1}}.
\end{split}
\end{align}
Denote the terms $\mathcal{H}(\z)\big((d\sqrt{\z})^2+\z h\big)+\big(\sum_{j=k+1}^{\infty}\frac{\mathscr{F}^{(j)}(0)}{j!}\z^{j+1}\big)g_{\CP^{n-1}}$ by $h.o.t.$, for higher order terms, since they are faster vanishing at the original than the first terms in the expansion.  Now, make a change in the radial variable by setting
\begin{align}
\label{r_change}
r^2=\Big(\frac{2n-1}{k+1}\Big)\Big(\frac{\mathscr{F}^{(k)}(0)}{k!}\Big)\z^{k+1}.
\end{align}
In these coordinates, the expansion \eqref{metric_expansion_1} of $g$ becomes
\begin{align}
\label{metric_expansion_2}
g_{(\mathscr{F},\mathcal{C})}=dr^2+\Big(\frac{k+1}{2n-1}\Big)r^2\Big(g_{\CP^{n-1}}+(2n-1)(k+1)h\Big)+h.o.t..
\end{align}
From \eqref{higher_order} and \eqref{coefficient_2_expansion}, we see that the higher order terms in these coordinates satisfy the decay conditions
\begin{align}
\label{hot_decay}
h.o.t.=\mathcal{O}\Big(r^{\frac{2}{k+1}}\Big)dr^2+\mathcal{O}\Big(r^{\frac{2(k+2)}{k+1}}\Big)h+\mathcal{O}\Big(r^{\frac{2(k+2)}{k+1}}\Big)g_{\CP^{n-1}}.
\end{align}
Clearly, this metric extends to all of $\CC^n$ and has the stated canonical asymptotically cone-like singularity at the origin.

\subsection{Regularity at the origin}
\label{regularity_proof}
Here, we examine the regularity, over the origin, of the coefficients of the higher order terms of the metric $g_{(\mathscr{F},\mathcal{C})}$ in the variable $r$.  There are two issues of which to be wary.  Firstly, there may be fractional powers of $r$ in the expansion of these coefficients when we change variables by \eqref{r_change} to obtain \eqref{metric_expansion_2} from \eqref{metric_expansion_1}.  Secondly, if the $\z^{(n-1)(k+1)}\log(\z)\sum_{j=0}^{\infty}H_j\z^j$ term in $\mathcal{G}(\z)$ is nonzero, it will cause derivatives of the coefficient of the $\big((d\sqrt{\z})^2+\z h\big)$ component of the metric to blow up in both $r$ and $\z$ variables, recall \eqref{G(z)}.

We first address the issue of fractional powers of $r$ in the expansion of these coefficients.  Ignoring the higher order terms that have coefficients involving $\log(\z)$, which will be dealt with below, up to constants, the other components of the higher order terms of $g_{(\mathscr{F},\mathcal{C})}$ are of the form $\z^a(d\sqrt{\z})^2$, $\z^bh$ and $\z^cg_{\CP^{n-1}}$ for some integers $a\geq k+1$ and $b,c\geq k+2$.
From the change of variables \eqref{r_change}, observe that for an arbitrary integer $d$
\begin{align}
\begin{split}
\label{z_l_hot}
\z^{d}=&\Big(\frac{2n-1}{k+1}\Big)^{-\frac{d}{k+1}}\Big(\frac{\mathscr{F}^{(k)}(0)}{k!}\Big)^{-\frac{d}{k+1}}r^{\frac{2d}{k+1}},\phantom{=}\text{and}\\
\z^{d}(d\sqrt{\z})^2=&(k+1)^{\frac{d-k}{k+1}}(2n-1)^{-\frac{k+d+2}{k+1}}\Big(\frac{\mathscr{F}^{(k)}(0)}{k!}\Big)^{-\frac{k+d+2}{k+1}}r^{\frac{2(d-k)}{k+1}}dr^2.
\end{split}
\end{align}
When $k=0$ or $k=1$, notice that there are no fractional powers of $r$ in the coefficients here, and the only regularity issue to consider is the higher order $\log(\z)$ term.  However, for $k\geq 2$ there may be fractional powers.  Observe, from \eqref{higher_order}, \eqref{coefficient_1_expansion} and \eqref{coefficient_2_expansion},
that the lowest possible such power necessarily comes from a $\z^{a}(d\sqrt{\z})^2$ term.   Since it is possible for $a=k+1$, from \eqref{z_l_hot}, we see that in the variable $r$ there may be an $r^{\frac{2}{k+1}}dr^2$ term which is only of class $C^0$ when $k\geq 2$.  Therefore, when $k\geq 2$, we can only guarantee, a priori, that $g_{(\mathscr{F},\mathcal{C})}$ is of class $C^0$ over the origin.  

However, when $k\geq 2$, there still may be higher regularity in a particular case, and if further information is known about the expansion of $\mathscr{F}$ it may be possible to improve this estimate.  For instance, if the order of the second nonzero term in the Taylor expansion of $\mathscr{F}(\z)$ about the origin is known, denote this by $\ell$, then the lowest possible order term with fractional powers of $r$ will be a constant multiple of $r^{\frac{2(\ell-k)}{k+1}}dr^2$.  If $\ell$ is such that $\frac{2(\ell-k)}{k+1}$ is an integer, then one would look a step further into the expansion, and so on and so forth.

Now, we address the issue of of the $\log(\z)$ term.  Observe, from \eqref{G(z)} and \eqref{1/zF_int}, that 
the lowest order term in the expansion of the coefficient of $\big((d\sqrt{\z})^2+\z h\big)$ that contains a $\log(\z)$ is a constant multiple of $\z^{(n-1)(k+1)+k}\log(\z)\sum_{j=0}^{\infty}H_j\z^j$.
Changing variables as in \eqref{r_change}, we see that
\begin{align}
\label{zlog_expansion}
\z^{(n-1)(k+1)+k}\log(\z)(d\sqrt{\z})^2=\frac{2}{C^n(k+1)^3}r^{2(n-1)}\log\Big(\frac{r}{\sqrt{C}}\Big)dr^2,
\end{align}
where the constant $C=\big(\frac{2n-1}{k+1}\big)\big(\frac{\mathscr{F}^{(k)}(0)}{k!}\big)$.  Recall, from \eqref{coefficient_1_expansion_0}, that if any of the $H_j$ are nonzero, then necessarily $H_0$ is nonzero.  Therefore, if $\sum_{j=0}^{\infty}H_j\z^j$ is nonzero, from \eqref{zlog_expansion}, we see that there are precisely $2n-3$
derivatives in the $r$ variable before the term containing $\log(r)$ blows up at the origin.  

The obstruction to regularity of the terms containing $\log(r)$ is removed if and only if $H_j=0$ for all $j$.  Observe, from \eqref{int_1/zG}, \eqref{int_1/zG^-1}, \eqref{D,E_coeffiecients} and \eqref{coefficient_1_expansion_0}, that this occurs if and only if 
\begin{align}
\frac{\partial^{(n-1)(k+1)}}{\partial\z^{(n-1)(k+1)}}\Big(\frac{\z^k}{\mathscr{F}}\Big)^{n-1}\Big|_{\z=0}=0.
\end{align}
This is equivalent to the $(n-1)(k+1)$-degree term of the Taylor expansion of $\big(\frac{\z^k}{\mathscr{F}}\big)^{n-1}$ around $\z=0$ vanishing.

In the cases that $k=0$ or $k=1$, since there are no fractional powers of $r$ in the expansion of the coefficients of $g_{(\mathscr{F},\mathcal{C})}$ as we saw above, this is the only obstruction to the coefficients of the metric extending smoothly over the origin, and the regularity results of Theorem~\ref{thm_3} follow.  When $k\geq2$, recall that, a priori, the metric is only guaranteed to be of class $C^0$ over the origin since there may be an $r^{\gamma}$ term, where $0<\gamma<1$, in a coefficient of the expansion of the higher order terms of the metric.  However, given more information about $\mathscr{F}(\z)$ and its expansion, the following precise statement can be made.  Let $\eta$ denote the lowest non-integral power of $r$ in the coefficients of the expansion of $g_{(\mathscr{F},\mathcal{C})}$.  Then, the metric is of class
\begin{align}
\begin{cases}
C^{\lfloor \eta\rfloor}\phantom{=}&\text{if }\frac{\partial^{(n-1)(k+1)}}{\partial\z^{(n-1)(k+1)}}\big(\frac{\z^k}{\mathscr{F}}\big)^{n-1}\big|_{\z=0}=0\\
C^{\min\{\lfloor \eta\rfloor, 2n-3\}}\phantom{=}&\text{otherwise}
\end{cases},
\end{align}
where $\lfloor\cdot\rfloor$ is the greatest integer function.
\begin{remark}
\label{k>2Cinfty}
{\em Note that, even when $k\geq 2$, it is possible that the coefficients of $g_{(\mathscr{F},\mathcal{C})}$ only contain integral powers of $r$.  In this case we say that $\eta=\infty$.  Then, the metric is of class $C^{\infty}$ if and only if $\frac{\partial^{(n-1)(k+1)}}{\partial\z^{(n-1)(k+1)}}\big(\frac{\z^k}{\mathscr{F}}\big)^{n-1}\big|_{\z=0}=0$.  For example, consider the admissible pair $(\mathscr{F},\mathcal{C})=(\z^2+\z^8,0)$ in complex dimension $n=2$.  The associated metric $g_{(\mathscr{F},\mathcal{C})}$ will be of class $C^{\infty}$ over the origin.  See Section \ref{fix_2}.}
\end{remark}

\subsection{Proof of Corollary \ref{cor_thm_3}}
\label{pf_cor_thm_3}
Consider the uniquely associated $\U(n)$-invariant non-K\"ahler Hermitian Klsc metric, $g_{(\mathscr{F},\mathcal{C})}$, to an admissible pair $(\mathscr{F},\mathcal{C})$ on the annulus $\mathcal{A}_{0,\infty}\subset(\CC^n,J_0)$, where $\mathscr{F}(\z)\in C^{\infty}(\CC^n)$ has first nonzero term of its Taylor expansion about the origin of order $k$.  From Theorem \ref{thm_3}, we see that this extends to a metric on $\CC^n$ with an asymptotically cone-like singularity at the origin.  
If $(2n-1)(k+1)=\ell^2$, for some integer $\ell$, then observe, from Theorem \ref{thm_3}, that this metric is given by
\begin{align}
g_{(\mathscr{F},\mathcal{C})}=dr^2+\Big(\frac{k+1}{2n-1}\Big)r^2\Big(g_{\CP^{n-1}}+\ell^2h\Big)+h.o.t..
\end{align}
Then, take the $\mathbb{Z}/\ell\mathbb{Z}$ quotient along the Hopf fibers generated by
\begin{align}
(z_1,\cdots,z_n)\mapsto (e^{2\pi i/\ell}z_1,\cdots,e^{2\pi i/\ell}z_n),
\end{align}
and note that this can be identified with $\CC^n$ itself.  Since the complex structure is compatible with this process, it descends to this quotient.  Thus, we see that the metric $g_{(\mathscr{F},\mathcal{C})}$ descends to 
\begin{align}
\overline{g}_{(\mathscr{F},\mathcal{C})}=dr^2+\Big(\frac{k+1}{2n-1}\Big)r^2g_{S^{2n-1}}+h.o.t.,
\end{align}
where $g_{S^{2n-1}}$ is the standard metric on $S^{2n-1}$.  Finally, notice that if $k+1=2n-1$, in which case $(2n-1)(k+1)$ is always a square, then
\begin{align}
\overline{g}_{(\mathscr{F},\mathcal{C})}=dr^2+r^2g_{S^{2n-1}}+h.o.t.
\end{align}
is a nonsingular metric.  These metrics are of the same regularity class over the origin as the initial metric $g_{(\mathscr{F},\mathcal{C})}$, recall Section \ref{regularity_proof}.
\begin{remark}
{\em An example of a nonsingular Klsc metric of class $C^{\infty}$ is obtained from the admissible pair $(\mathscr{F},\mathcal{C})=(\z^2+\z^8,0)$ in complex dimension $n=2$, since $3=2n-1=k+1$, and from Remark \ref{k>2Cinfty} we know that $\overline{g}_{(\mathscr{F},\mathcal{C})}$ is of class $C^{\infty}$ over the origin.  See Section \ref{fix_2}.
}
\end{remark}

\section{Examples}
In this section we give several examples of non-K\"ahler Hermitian Klsc metrics arising via Theorem~\ref{thm_1}, Theorem \ref{thm_2} and Theorem \ref{thm_3}.  In Section \ref{flat}, Section~\ref{burns} and Section \ref{Fubini-Study}, we begin with a K\"ahler metric and obtain the corresponding non-K\"ahler Hermitian Klsc metric via Theorem \ref{thm_1}.  We also find the scalar curvature of the resulting Klsc metric in each case and examine its behavior.
Then, in Section~\ref{fix_1} and Section \ref{fix_2}, we fix an admissible pair $(\mathscr{F},\mathcal{C})$ on the annulus $\mathcal{A}_{0,\infty}$, and find the unique associated non-K\"ahler Hermitian Klsc metric.  We also examine the behavior and regularity of the respective Klsc metrics at the origin as well as find the K\"ahler metrics to which they are conformal.  
To simplify the computations, we will work in in complex dimension $n=2$, except in Section \ref{flat}, as well as take the constant of the admissible pair to be $\mathcal{C}=0$ in Section \ref{fix_1} and Section \ref{fix_2}.

\subsection{The non-K\"ahler Klsc metric conformal to the Euclidean metric on $\CC^n$}
\label{flat}
Consider the flat metric, $g_{Euc}=(d\sqrt{\z})^2+\z g_{S^{2n-1}}$, on $\CC^n$ which arises from the K\"ahler potential $\phi=\z$.  From Theorem \ref{thm_1}, the conformal metric $g_{Klsc}=v^2g_{Euc}$, where
\begin{align}
\label{flat_u}
v^2=(n-1)^{-\frac{2}{2n-1}}\z^{-\frac{2(n-1)}{2n-1}},
\end{align}
is a non-K\"ahler Hermitian Klsc metric.
Changing variables by setting 
\begin{align}
r^2=(n-1)^{-\frac{2}{2n-1}}(2n-1)^2\z^{\frac{1}{2n-1}},
\end{align}
we see that the Klsc metric, given by
\begin{align}
g_{Klsc}=dr^2+\frac{r^2}{(2n-1)^2}g_{S^{2n-1}},
\end{align}
is a conical metric. 

Lastly, we find that the scalar curvature of this Klsc metric $g_{Klsc}$ is given by
\begin{align}
S_{g_{Klsc}}=2S_{C_{g_{Klsc}}}=8n(n-1)^{\frac{4n}{2n-1}}\z^{-\frac{1}{2n-1}}=8n(2n-1)(n-1)^2r^{-2}.
\end{align}

\subsection{The non-K\"ahler Klsc metric conformal to the Burns metric on $\mathcal{O}_{\PP^1}(-1)$}
\label{burns}
The Burns metric is a scalar-flat K\"ahler ALE metric on the blow-up of $\CC^2$ at the origin \cite{Burns}.  It arises from the K\"ahler potential $\phi=\z+\log(\z)$, and is given by
\begin{align}
g_B=(d\sqrt{\z})^2+(\z+1)g_{\CP^1}+\z h.
\end{align}
Note the $\CP^1$ at the origin.  From Theorem \ref{thm_1}, we see that the conformal metric $g_{Klsc}=v^2g_{B}$, where
\begin{align}
\label{flat_burns}
v^2=\log^{\frac{2}{3}}\Big(1+\frac{1}{\z}\Big),
\end{align}
is a non-K\"ahler Hermitian Klsc metric.  Explicitly, we have that 
\begin{align}
g_{Klsc}=\log^{\frac{2}{3}}\Big(1+\frac{1}{\z}\Big)\Big((d\sqrt{\z})^2+(\z+1)g_{\CP^1}+\z h\Big).
\end{align}
Notice that here, the Klsc metric $g_{Klsc}$ is defined on $\CC^2\setminus\{0\}$.  In a sense, one can view the conformal factor as causing the metric to blow up along the $\CP^1$ at the origin while the Hopf fibers still shrink.

Lastly, we find that the scalar curvature of this Klsc metric $g_{Klsc}$ is given by
\begin{align}
S_{g_{Klsc}}=2S_{C_{g_{Klsc}}}=\frac{8}{3}\Bigg(\z(\z+1)^2\log^{\frac{8}{3}}\Big(1+\frac{1}{\z}\Big)\Bigg)^{-1}.
\end{align}
Notice that the scalar curvature blows as the origin is approached.

\subsection{The non-K\"ahler Klsc metric conformal to the Fubini-Study metric}
\label{Fubini-Study}
Here we will see Theorem \ref{thm_1} applied to a compact manifold and obtain non-K\"ahler Klsc metrics on an annular decomposition of the manifold minus closed sets of positive codimension.  Consider the Fubini-Study metric, which is K\"ahler Einstein, on $\CP^2$.  It arises from the K\"ahler potential $\phi=\log(\z+1)$, and is given by
\begin{align}
g_{FS}=\frac{1}{(\z+1)^2}\Big((d\sqrt{\z})^2+\z h\Big)+\frac{\z}{(\z+1)}g_{\CP^1}.
\end{align}
Clearly, this extends over the point at $\z=0$ and the $2$-sphere at infinity smoothly.  From Theorem \ref{thm_1}, we see that the conformal metric $g_{Klsc}=v^2g_{FS}$, where
\begin{align}
v^2=\Big|\log(\z)-\frac{1}{\z}\Big|^{\frac{2}{3}},
\end{align}
is a non-K\"ahler Klsc metric wherever it is defined.  To see where this is, first notice that $\log(\z)-\frac{1}{\z}=0$ at $\z=\gamma\approx 1.763\dots$.  Next, observe that the conformal factor causes the metric to blow up along to $2$-sphere at infinity, while the Hopf fibers still shrink there.  Therefore, removing the $2$-sphere at infinity, the point at the origin and the $S^3$ given by $\z=\gamma$, we find that the conformal metric
\begin{align}
g_{Klsc}=\Big|\log(\z)-\frac{1}{\z}\Big|^{\frac{2}{3}}\Bigg(\frac{1}{(\z+1)^2}\Big((d\sqrt{\z})^2+\z h\Big)+\frac{\z}{(\z+1)}g_{\CP^1}\Bigg)
\end{align}
is a non-K\"ahler Klsc metric on the disjoint union of annuli $\mathcal{A}_{0,\gamma}\amalg\mathcal{A}_{\gamma,\infty}$.

Lastly, we find that the scalar curvature of this Klsc metric $g_{Klsc}$.  Although it is necessary to separate the computation of the transformation of scalar curvature on each of the respective annuli since $\log(\z)-\frac{1}{\z}$ is negative for $\z<\gamma$ and positive for $\z>\gamma$, we find that the same formula gives the scalar curvature of the non-K\"ahler Klsc metric on both $\mathcal{A}_{0,\gamma}$ and $\mathcal{A}_{\gamma,\infty}$, and this is
\begin{align}
S_{g_{Klsc}}=2S_{C_{g_{Klsc}}}=\Big(\log(\z)+\frac{1}{\z}\Big)^{-\frac{2}{3}}\Bigg(\frac{8}{3}(\z+1)\Big(1+\frac{1}{\z}\Big)^3\Big(\log(\z)+\frac{1}{\z}\Big)^{-1}+12\Bigg).
\end{align}

\subsection{Obtaining the non-K\:ahler Klsc metric by fixing $\mathscr{F}(\z)=\z^2+\z^8$ on $\CC^2$}
\label{fix_2}
Since $(\z^2+\z^8)'>0$ for all $\z>0$, the pair 
\begin{align}
(\mathscr{F},\mathcal{C})=(\z^2+\z^8,0)
\end{align}
is clearly admissible on the annulus $\mathcal{A}_{0,\infty}\subset\CC^2$.
Note that the first nonzero term of the Taylor expansion of $\mathscr{F}$ about the origin is of order $2$.  To obtain the corresponding Klsc metric, from Theorem \ref{thm_2} and Theorem \ref{thm_3}, in a small neighborhood of the origin we compute that 
\begin{align}
\begin{split}
(\z\mathscr{F})'-\frac{2}{\z}\Big(\int\frac{1}{\z^2\mathscr{F}}d\z\Big)^{-1}= & 9\z^2+9\z^8+6\z^2\sum_{n=1}^{\infty}\Big(\sum_{j=1}^{\infty}\frac{(-1)^j}{2j-1}\z^{6j}\Big)^n.
\end{split}
\end{align}
Thus, in a small neighborhood of the origin,the Klsc metric is given by
\begin{align}
\begin{split}
g_{(\mathscr{F},0)}=\Bigg(9\z^2+9\z^8+6\z^2\sum_{n=1}^{\infty}\Big(\sum_{j=1}^{\infty}\frac{(-1)^j}{2j-1}\z^{6j}\Big)^n\Bigg)\Big((d\sqrt{\z})^2+\z h\Big)+(\z^3+\z^9)g_{\CP^1}.
\end{split}
\end{align}
Then, by making the change of variables $r^2=\z^3$, we see that
\begin{align}
g_{(\mathscr{F},0)}=dr^2+r^2(g_{\CP^1}+9h)+\Bigg(r^4+\frac{6}{9}\sum_{n=1}^{\infty}\Big(\sum_{j=1}^{\infty}\frac{(-1)^j}{2j-1}r^{4j}\Big)^n\Bigg)(dr^2+9r^2h)+r^6g_{\CP^1}.
\end{align}
Notice that there are no fractional powers of $r$ in the coefficients of the metric.  Therefore, recalling 
Remark \ref{k>2Cinfty}, even though $k=2$, this metric is of class $C^{\infty}$ over the origin since 
\begin{align}
\frac{\partial}{\partial\z}\Big(\frac{\z^2}{\mathscr{F}}\Big)\Big|_{\z=0}=0.
\end{align}

Now, as in Corollary \ref{cor_thm_3}, the metric descends via the $\mathbb{Z}/3\mathbb{Z}$ quotient along the Hopf fibers generated by
\begin{align}
(z_1,z_2)\mapsto(e^{\frac{2\pi i}{3}}z_1,e^{\frac{2\pi i}{3}}z_2)
\end{align}
to the nonsingular Klsc metric given by
\begin{align}
\begin{split}
\overline{g}_{(\mathscr{F},0)}=dr^2+r^2(g_{\CP^1}+h)+\Bigg(r^4+\frac{6}{9}\sum_{n=1}^{\infty}\Big(\sum_{j=1}^{\infty}\frac{(-1)^j}{2j-1}r^{4j}\Big)^n\Bigg)(dr^2+r^2h)+r^6g_{\CP^1}.
\end{split}
\end{align}
Furthermore, note that this is in fact a smooth nonsingular metric.

Lastly, we are able to recover the K\"ahler metric to which the Klsc metric $g_{(\mathscr{F},0)}$ is conformal.  From Proposition \ref{admissible_proposition}, we find that the derivative of the K\"ahler potential corresponding to this Klsc metric is
\begin{align}
\phi'=(\z^2+\z^8)\Big(\frac{1}{3}\int\frac{1}{\z^2(\z^2+\z^8)}d\z\Big)^{-2}=\frac{81(\z^2+\z^8)}{\big(\tan^{-1}(\frac{1}{\z^3})-\frac{1}{\z^3}\big)^2}.
\end{align}
Then, using this to compute $(\z\phi')'$, we can recover the K\"ahler form restricted to the $z_1$-axis and, in turn, find that the K\"ahler metric is given by
\begin{equation}
g=\frac{243\big((3\z^9+\z^3)\tan^{-1}\big(\frac{1}{\z^3}\big)-3(\z^6+1)\big)}{\z\big(\tan^{-1}(\frac{1}{\z^3})-\frac{1}{\z^3}\big)^3}\Big((d\sqrt{\z})^2+\z h\Big)+\frac{81(\z^3+\z^9)}{\big(\tan^{-1}(\frac{1}{\z^3})-\frac{1}{\z^3}\big)^2} \big(g_{\CP^1}\big).
\end{equation}

\subsection{Obtaining the non-K\:ahler Klsc metric by fixing $\mathscr{F}(\z)=\z+\z^2$ on $\CC^2$}
\label{fix_1}
While this choice of specified function may look similar to that of Section \ref{fix_2}, it will illustrate a different type of possible behavior in that it will only be of class $C^1$ over the origin as opposed to class $C^{\infty}$.  Furthermore, it will not descend to a nonsingular metric.

Since $(\z+\z^2)'>0$ for all $\z>0$, the pair 
\begin{align}
(\mathscr{F},\mathcal{C})=(\z+\z^2,0)
\end{align}
is clearly admissible on the annulus $\mathcal{A}_{0,\infty}\subset\CC^2$.
Note that the first nonzero term of the Taylor expansion of $\mathscr{F}$ about the origin is of order $1$.  To obtain the corresponding Klsc metric, from Theorem \ref{thm_2} and Theorem \ref{thm_3}, we compute that 
\begin{align}
\begin{split}
(\z\mathscr{F})'-\frac{2}{\z}\Big(\int\frac{1}{\z^2\mathscr{F}}d\z\Big)^{-1}= 2\z+3\z^2+\frac{4\z}{1-2\z+2\z^2\log(1+\frac{1}{\z})}.
\end{split}
\end{align}
Thus, the Klsc metric is given by
\begin{align}
\begin{split}
g_{(\mathscr{F},0)}=\Bigg(2\z+3\z^2+\frac{4\z}{1-2\z+2\z^2\log(1+\frac{1}{\z})}\Bigg)\Big((d\sqrt{\z})^2+\z h\Big)+(\z^2+\z^3)g_{\CP^1}.
\end{split}
\end{align}
Then, by making the change of variables $r^2=\frac{3}{2}\z^2$ and taking the appropriate expansions in a small neighborhood of the origin, we see that
\begin{align}
\begin{split}
g_{(\mathscr{F},0)}=&dr^2+\frac{2}{3}r^2(g_{\CP^1}+6h)+\Big(\frac{r}{\sqrt{6}}+\frac{2}{3}\sum_{n=1}^{\infty}H(r)^n\Big)(dr^2+4r^2h)+\Big(\frac{2}{3}\Big)^{\frac{3}{2}}r^3g_{\CP^1},
\end{split}
\end{align}
where 
\begin{align}
H(r)=\big(\sqrt{8/3}\big)r-\frac{4}{3}r^2\log\Big(1+\big(\sqrt{3/2}\big)r^{-1}\Big).
\end{align}
Notice that while there are no fractional powers of $r$ in the coefficients of the metric there are, however, nonvanishing $\log(\cdot)$ terms.  Without actually having found this expansion explicitly, we could have seen that this occurs from Theorem \ref{thm_3} since
\begin{align}
\frac{\partial^2}{\partial\z^2}\Big(\frac{\z}{\mathscr{F}}\Big)\Big|_{\z=0}=2\neq0.
\end{align}
In turn, this metric is only of class $C^1$ over the origin.

Lastly, we are able to recover the K\"ahler metric to which this Klsc metric $g_{(\mathscr{F},0)}$ is conformal.  From Proposition \ref{admissible_proposition}, we find that the derivative of the K\"ahler potential corresponding to this Klsc metric is
\begin{align}
\phi'=(\z^2+\z^2)\Big(\frac{1}{3}\int\frac{1}{\z^2(\z^2+\z^2)}d\z\Big)^{-2}=\frac{36(\z^3+\z^4)}{\big(1-2\z+2\z^2\log(1+\frac{1}{\z})\big)^2}.
\end{align}
Then, using this to compute $(\z\phi')'$, we can recover the K\"ahler form restricted to the $z_1$-axis and, in turn, find that the K\"ahler metric is given by
\begin{align}
\begin{split}
g=\frac{36\z^5\big(6-\z-6\z^2+2(3\z+2)\z^2\log(1+\frac{1}{\z})\big)}{\big(1-2\z+2\z^2\log(1+\frac{1}{\z})\big)^3}&\Big((d\sqrt{\z})^2+\z h\Big)\\
+\frac{36(\z^4+\z^5)}{\big(1-2\z+2\z^2\log(1+\frac{1}{\z})\big)^2} &\big(g_{\CP^1}\big).
\end{split}
\end{align}

\bibliography{hermitian_scalar_references}

\end{document}